\documentclass[11pt,dvipsnames]{amsart}

\usepackage{amsmath,amssymb,amsthm,tikz, nicefrac, mathrsfs,upgreek}

\RequirePackage{color}
\RequirePackage[colorlinks, urlcolor=my-blue,linkcolor=my-red,citecolor=my-red]{hyperref}
\definecolor{my-blue}{rgb}{0.0,0.0,0.6}
\definecolor{my-red}{rgb}{0.5,0.0,0.0}
\definecolor{my-green}{rgb}{0.0,0.5,0.0}
\definecolor{nicos-red}{rgb}{0.75,0.0,0.0}
\definecolor{nicos-green}{rgb}{0.0,0.75,0.0}
\definecolor{light-gray}{gray}{0.6}
\definecolor{really-light-gray}{gray}{0.8}
\definecolor{sussexg}{rgb}{0.0,0.5,0.5}
\definecolor{sussexp}{rgb}{0.5,0.0,0.5}

\usepackage{dsfont} 

\usepackage[utf8]{inputenc}

\usepackage{nicefrac}

\usepackage{multicol}
\PassOptionsToPackage{dvipsnames}{xcolor}
\usepackage{tikz,xcolor}

\usetikzlibrary{shapes,arrows}
\usetikzlibrary{hobby}
\usetikzlibrary{decorations.markings}

\newtheorem{theorem}{\color{my-red}{\sc Theorem}}[section]
\newtheorem{lemma}[theorem]{\color{my-red} \sc Lemma}
\newtheorem{proposition}[theorem]{\color{my-red} \sc Proposition}
\newtheorem{corollary}[theorem]{\color{my-red} \sc Corollary}

\numberwithin{equation}{section}
\theoremstyle{remark}
\newtheorem{remark}[theorem]{\color{my-red} Remark}

\newcommand{\be}{\begin{equation}}
\newcommand{\ee}{\end{equation}}

\def\hei{T}

\def\bE{\mathbb{E}}
\def\bN{\mathbb{N}}
\def\bP{\mathbb{P}}
\def\bR{\mathbb{R}}
\def\bZ{\mathbb{Z}}

\def\DAGp{\textsf{\textup{DAG}. }}
\def\DAGc{\textsf{\textup{DAG}, }}
\def\DAGsc{\textsf{\textup{DAG}}s, }
\def\DAGsp{\textsf{\textup{DAG}}s. }

\def\MECp{\textsf{\textup{MEC}. }}
\def\MECc{\textsf{\textup{MEC}, }}

\def\PDAG{\textsf{\textup{PDAG} }}
\def\PDAGs{\textsf{\textup{PDAG}}s }
\def\DAG{\textsf{\textup{DAG} }}
\def\MEC{\textsf{\textup{MEC} }}
\def\DAGs{\textsf{\textup{DAG}}s }
\def\MECs{\textsf{\textup{MEC}}s }

 \def\Z{\bZ}
\def\Q{\mathbb{Q}}
\def\R{\bR}
\def\N{\bN}
\def\P{\bP}

\usepackage{stackengine}

\def\E{\bE}
\def\P{\bP} 

\definecolor{partcolor1}{rgb}{0.0,0.5,0.0}
\definecolor{partcolor2}{rgb}{0.0,0.5,0.0}

\definecolor{darkgreen}{rgb}{0.0,0.5,0.0}
\definecolor{darkblue}{rgb}{0.5,0.1,0.5}
\definecolor{deepblue}{rgb}{0.25,0.41,0.88}
\definecolor{nicosred}{rgb}{0.65,0.1,0.1}
\definecolor{light-gray}{gray}{0.7}
\allowdisplaybreaks[1]

\RequirePackage{datetime} 
\allowdisplaybreaks
\begin{document}
\usdate
\title[On the number and size of \MECs of \DAGs]
{On the number and size of Markov equivalence classes of random directed acyclic graphs}
\author{Dominik Schmid}
\address{Dominik Schmid, University of Bonn, Germany}
\email{d.schmid@uni-bonn.de}
\author{Allan Sly}
\address{Allan Sly, Princeton University, United States}
\email{asly@princeton.edu}
\date{\today}
\begin{abstract} 

In causal inference on directed acyclic graphs, the orientation of edges is in general only recovered up to Markov equivalence classes.
We study Markov equivalence classes of uniformly random directed acyclic graphs. Using a tower decomposition, we show that the ratio between the number of Markov equivalence classes and directed acyclic graphs approaches a positive constant when the number of sites goes to infinity. For a typical directed acyclic graph,  the expected number of elements in its Markov equivalence class remains bounded. More precisely, we prove that for a uniformly chosen directed acyclic graph, 
the size of its Markov equivalence class has super-polynomial tails.
\end{abstract}
\maketitle
\vspace*{-0.8cm}
\section{Introduction} \label{sec:Introduction}
For a graphical model on a directed acyclic graph, the conditional independence information which can be extracted from the model, is encoded by its $v$-structures, a collection of $v$-shaped subgraphs. Given the underlying undirected graph, and the $v$-structures, is it possible to reconstruct all edge orientations? In general, having only access to the conditional independence between sites,  and the set of undirected edges, is not sufficient to uniquely determine a directed acyclic graph (\textsf{DAG}).   We refer to its \textbf{Markov equivalence class} (\textsf{MEC}) as the set of all \DAGs with the same undirected graph, and the same $v$-structures. Markov equivalence classes are a central object in causality theory with applications in machine learning, economics and social sciences \cite{C:LearningMECs,EGS:Learning,ES:Philosophy, FLGZH:LocalMEC,P:Causality}. \\

While the size of a Markov equivalence class of a \DAG on $n$ sites can be as large as $n!$, for example when the graph is a clique, simulations suggest that the ratio between Markov equivalence classes and directed acyclic graphs on $n$ sites remains uniformly bounded from below when $n$ goes to infinity~\cite{S:EnumerateLabeled}. 
Similarly, the ratio between \textbf{essential directed acyclic graphs}, which are the unique elements in their Markov equivalence class, and the total number of \DAGs seems also to be uniformly bounded in $n$. 
Over the past decades, exact recursive descriptions for counting the number of \DAGsc as well as the number of essential \DAGsc were found \cite{RL:CountingUnlabeled,S:EnumerateLabeled,S:EnumerateEssential}. However, there are no simple exact formulas available to give a precise asymptotic control on the above ratios. In this paper, we show that the ratio between \MECs and \DAGs on $n$ sites, as well as the ratio between essential directed acyclic graphs and \DAGsc both converge to some positive constants $c_{\textup{MEC}}$ and $c_{\textup{Ess}}$, respectively, when $n$ goes to infinity. In particular, with positive probability, a \DAG chosen uniformly at random on $n$ sites is the unique element in its Markov equivalence class. To our best knowledge, these are the first rigorous results on the above ratios for this model.  

In recent years, the question of counting of the number of Markov equivalence classes for various families of \DAGsc including \DAGs on trees and complete graphs with a few missing edges, is studied in the machine learning community \cite{C:LearningMECs,EGS:Learning,FLGZH:LocalMEC,HJY:ExploreSize,TK:CountingDAG}. We show that the expected size of a Markov equivalence class for a directed acyclic graph, chosen uniformly at random on $n$ sites, is bounded uniformly in $n$. In contrast to counting the number of \MECs on a given underlying graph, the challenge for a uniform sampled \DAG is that we do not have  access to a simple recursive sampling procedure, and thus have to directly study its typical structure in order to get bounds on its Markov equivalence class.  
\subsection{Model and results}\label{sec:ModelAndResults}

Let $G$ be a directed acyclic graph with (labeled) vertex set $V$,  with $V=[n]:= \{1,\dots,n\}$ for some $n\in \N$, and directed edge set $E$. We denote by $\mathcal{D}_n$ the set of all labeled \DAGs on $n$ vertices. 
For any triplet of vertices $(a,b,c)\in V^3$, we say that $c$ is a \textbf{collider} if $(a,c),(b,c) \in E$, while $(a,b),(b,a) \notin E$. In this case,  we call $(a,b,c)$ a $\boldsymbol v$\textbf{-structure}, also called intervention or immorality, at $c$. 
For a fixed \DAG $G \in \mathcal{D}_n$, we denote by $M(G)$ the set of all \DAGs which have the same edge set when forgetting the edge directions, also called the \textbf{skeleton} of $G$, and the same $v$-structures as $G$. We refer to $M(G)$ as the \textbf{Markov equivalence class} of $G$; see also Figure \ref{fig:DAGandMEC}. In the following, our main focus is to estimate the number and size of Markov equivalence classes for directed acyclic graphs sampled uniformly at random on a given number of sites. \\

For a uniformly sampled \DAG $G$ on $n$ sites, let $\P_n$ and $\E_n$ be the corresponding law and expectation. We have the following main result on the tails of the size of a Markov equivalence class for a \DAG $G$ sampled according to $\P_n$.

\begin{theorem}\label{thm:Moments} For all $n\in \N$ and $t \geq 0$ sufficiently large, we have that
\begin{equation}\label{eq:TailsMain}
\P_{n}( |M(G)|> t) \leq C \exp\big(-\log(t)^{1+\alpha}\big)
\end{equation} for some constant $C>0$ and $\alpha=\frac{1}{20}$. 
Thus, for every $k\in \N$, there exists some constant $c_k>0$ such that for all $n\in \N$
\begin{equation}\label{eq:MomentsMain}
\E_n\Big[ |M(G)|^k \Big] \leq c_k \, .
\end{equation} In particular, the expected size of a \MEC for a uniform \DAG on $n$ sites is finite.
\end{theorem}
The constant $\alpha$ in \eqref{eq:TailsMain} is not optimized, but can be at most $1$, e.g.\ when we set $t=n!-1$ and consider the event that the underlying skeleton is a clique. It remains an open question whether the parameter $\alpha$ in \eqref{eq:TailsMain} can be chosen arbitrarily close to $1$. Next, we consider the set of \textbf{non-collider edges} $C(G)$, i.e.\ the edges which are not part of any $v$-structure, and show that the number of non-collider edges $|C(G)|$ converges to some random quantity. 
\begin{theorem}\label{thm:Essential} There exists some positive, almost surely finite random variable $X$ with
\begin{equation}\label{eq:CountingNonColliders}
\lim_{n \to \infty} \P_{n}( |C(G)| \geq s) = \P(X \geq s)
\end{equation} for all $s\geq 0$. Moreover, we have that
\begin{equation}\label{eq:SeeingNonColliders}
\lim_{n \to \infty} \P_{n}( |C(G)| = 0) > 0 \, ,
\end{equation} i.e.\ with positive probability, a \DAG according to $\P_n$ is the unique element in its \MECp
\end{theorem}

When a given \DAG is the unique element in its Markov equivalence class, we refer to it as an \textbf{essential directed acyclic graph}. We have the following bounds on the number of Markov equivalence classes, and the number of essential directed acyclic graphs on $n$ sites.
\begin{theorem}\label{thm:AverageEssentialMEC} 
Let $\mathcal{M}_n$ be the set of all distinct \MECs on $n$ sites. Then we have that 
\begin{equation}\label{eq:AverageMEC}
 \lim_{n \to \infty} \frac{|\mathcal{M}_n|}{|\mathcal{D}_n|} = c_{\textup{MEC}} 
 \end{equation} for some universal constant $c_{\textup{MEC}}>0$.
Let $\mathcal{M}^{\textup{Ess}}_{n} \subseteq \mathcal{M}_n$ be the number of Markov equivalence classes corresponding to essential \DAGsp Then we have that
\begin{equation}\label{eq:AverageEssential}
 \lim_{n \to \infty} \frac{|\mathcal{M}^{\textup{Ess}}_n|}{|\mathcal{D}_n|} = c_{\textup{Ess}} 
 \end{equation} for some universal constant $c_{\textup{Ess}}>0$.
\end{theorem}

Simulations suggest that $c_{\textup{Ess}}^{-1} \approx 13.65$, while the ration between the number of Markov equivalence class to \DAGs satisfies $c^{-1}_{\textup{MEC}} \approx 3.94$; see \cite{S:EnumerateLabeled,S:EnumerateEssential}, and \cite{BT:NoteMECS} for a heuristic explanation using linear extensions of the reachablity poset of a \DAGp 

\begin{remark} Let us stress that the above results are proven with respect to  labeled directed acyclic graphs. In \cite{BR:Asymptotic2}, it is shown that a \DAG drawn according to $\P_n$ has only trivial automorphisms with probability tending to $1$ when $n$ goes to infinity. Thus, Theorem~\ref{thm:Essential} and Theorem~\ref{thm:AverageEssentialMEC} extend  to sampling an unlabeled \DAG uniformly at random.
\end{remark}

\begin{figure} 
\begin{center}
\begin{tikzpicture}[scale=0.45]

	\node[shape=circle,scale=0.9,draw] (A) at (0,0.5){$2$} ;
 	\node[shape=circle,scale=0.9,draw] (B) at (-2.5,2){$7$} ;
	\node[shape=circle,scale=0.9,draw] (C) at (2.5,2){$1$} ;
 	\node[shape=circle,scale=0.9,draw] (D) at (0,3.5){$4$} ;	
 	\node[shape=circle,scale=0.9,draw] (E) at (0,6){$3$} ;
	\node[shape=circle,scale=0.9,draw] (F) at (-2,8){$5$} ;
 	\node[shape=circle,scale=0.9,draw] (G) at (2,8){$6$} ;

 	  	\draw[line width=1.7,->] (A) to (B);		
 	  	\draw[line width=1.7,->] (A) to (C);	
 	  	\draw[line width=1.7,->,my-red,densely dotted] (B) to (D);		
 	  	\draw[line width=1.7,->,my-red,densely dotted] (C) to (D);	 	
	  	\draw[line width=1.7,->] (D) to (E);		
	  	\draw[line width=1.7,->] (E) to (F);	
	  	\draw[line width=1.7,->,my-red,densely dotted] (E) to (G);			
	  	\draw[line width=1.7,->,my-red,densely dotted] (C) to (G);	
	  	\draw[line width=1.7,->] (B) to (E);	
 
 \def\x{9};

	\node[shape=circle,scale=0.9,draw] (A) at (0+\x,0.5){$2$} ;
 	\node[shape=circle,scale=0.9,draw] (B) at (-2.5+\x,2){$7$} ;
	\node[shape=circle,scale=0.9,draw] (C) at (2.5+\x,2){$1$} ;
 	\node[shape=circle,scale=0.9,draw] (D) at (0+\x,3.5){$4$} ;	
 	\node[shape=circle,scale=0.9,draw] (E) at (0+\x,6){$3$} ;
	\node[shape=circle,scale=0.9,draw] (F) at (-2+\x,8){$5$} ;
 	\node[shape=circle,scale=0.9,draw] (G) at (2+\x,8){$6$} ;

 	  	\draw[line width=1.7,->] (B) to (A);		
 	  	\draw[line width=1.7,->] (A) to (C);	
 	  	\draw[line width=1.7,->,my-red,densely dotted] (B) to (D);		
 	  	\draw[line width=1.7,->,my-red,densely dotted] (C) to (D);	 	
	  	\draw[line width=1.7,->] (D) to (E);		
	  	\draw[line width=1.7,->] (E) to (F);	
	  	\draw[line width=1.7,->,my-red,densely dotted] (E) to (G);			
	  	\draw[line width=1.7,->,my-red,densely dotted] (C) to (G);	
	  	\draw[line width=1.7,->] (B) to (E);

\end{tikzpicture}
\end{center}
\caption{\label{fig:DAGandMEC} Example of  two \DAGs on $n=7$ sites in the same \MECc differing on $\{2,7\}$. Edges in dotted red are part of a $v$-structure. Note that the edge $(4,3)$ is present for all graphs in the \MEC despite not being part of a collider.} 
\end{figure}
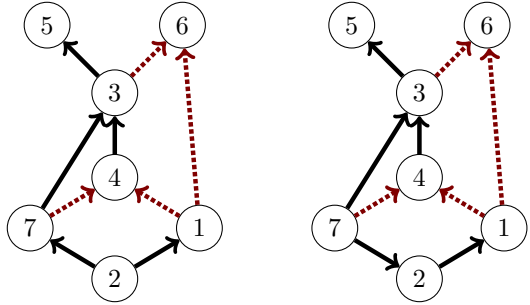

\subsection{Related work}\label{sec:RelatedWork}

Directed acyclic graphs and their Markov equivalence classes are  central objects in causality theory. In a more general context, the concept of $d$-separation is used for graphical models to describe conditional independence; see \cite{P:Causality} for a classical introduction to causality theory, and \cite{KF:GraphicalModels} for an exhaustive treatment of graphical models. In our setup, the special case of a Bayesian network, where the underlying graph is directed and acyclic, conditional independence is described by $v$-structures. Reconstructing a \DAG from its skeleton and $v$-structure, respectively estimating the size of its Markov equivalence class, has a variety of applications in machine learning and related disciplines; see  \cite{C:LearningMECs,EGS:Learning,ES:Philosophy,FLGZH:LocalMEC}. 

Over the last decades, various results on describing and counting directed acyclic graphs are achieved, see for example  \cite{S:EnumerateLabeled,S:EnumerateEssential} for recursive formulas on counting labeled (essential) \DAGsc as well as \cite{BRRW:Asymptotic1,BR:Asymptotic2,RL:CountingUnlabeled} for enumerations when the number of edges is fixed, or the \DAGs are unlabeled. These formulas allowed also for counting certain statistics of \DAGsc see for example \cite{L:MaximalVertices} on the number of sources of a \DAG chosen uniformly at random on a given number of sites. However, no explicit formulas to express the number of elements in the Markov equivalence class of general  directed acyclic graphs are currently known. \\

An alternative characterization of Markov equivalence classes is given by Andersson et al.\ using partially directed acyclic graphs, i.e.\ where some edges in the \DAG can be undirected~\cite{AMP:MECsForDAG}. Each Markov equivalence class corresponds to some partially directed acyclic graph -- see Section \ref{sec:RepresentationOfMECs} -- and Andersson et al.\ provide necessary and sufficient conditions for a partially directed acyclic graphs to correspond to a Markov equivalence class.  \\

In recent years, there is progress on computing the number of Markov equivalence classes for a fixed underlying skeleton, e.g.\ when the underlying graph is a tree or the complete graph with a few missing edges, see \cite{AG:EnumerateMECs,G:CountingSource,G:CountingMECs,
GP:Enumerate,HJY:ExploreSize,KSSU:SizeMEC,
RSU:CountingImmortal,RSU:MECCountTree} for a non-exhaustive list of results in this direction. Moreover, sampling the number and size of Markov equivalence classes, in particular using Markov chains, is a main question of interest; see \cite{BT:NoteMECS,GHT:EffectiveAlgo,GSKZ:CountingClique,
 GP:SizeDistribution,TK:CountingDAG}, and also \cite{GPV:SamplingUniformly,WBL:Polynomial} for very recent results on exact polynomial-time sampling of Markov equivalence classes on general \DAGsp The above results have in common that they rely on knowledge about the structure on the directed acyclic graph, either in terms of its skeleton, or a recursive decomposition technique such as clique trees. \\

When estimating the number and size of Markov equivalence classes for a directed acyclic graphs chosen uniformly at random on a given number of sites, the set of applicable techniques is more restricted. Exact recursive formulas for counting and a decomposition via clique trees are available, but difficult to control. In \cite{BT:NoteMECS}, a heuristic on the number of \MECs and essential \DAGs is given using the reachability poset of a directed acyclic graph. In this paper, we rely on a similar so-called tower decomposition, which can be found for example in \cite{BR:Asymptotic2} and \cite{M:ShapeDAG} to study the structure of a typical \DAGsc i.e.\ when removing the assumption of labeled sites, or when conditioning on a fixed number of edges in the graph. These structural results rely crucially on generating functions, and we will use them as a starting point for more probabilistic arguments to study Markov equivalence classes.

\subsection{Outline of the paper}\label{sec:Outline}

In order to show our main results, we pursue the following strategy. We start in Section~\ref{sec:HeightDecomposition} by recalling the tower decomposition for directed acyclic graphs. In Section~\ref{sec:DecompositionWeights}, we give precise bounds on the number of direct acyclic graphs corresponding to a given tower vector. To this end, we introduce the notion of regeneration points, which allows us in Section~\ref{sec:DecompositionOfTowers} to refine previously known results on the shape of towers. With this at hand, we prove convergence of the law of the last layers in a tower decomposition. In return, this allows us to establish Theorems~\ref{thm:Essential} and \ref{thm:AverageEssentialMEC} in Section~\ref{sec:NumberColliders} on the convergence of the number of non-collider edges.
In Section~\ref{sec:RepresentationOfMECs}, we collect basic estimates on the number of elements in a Markov equivalence class, and we present a refined counting method for Markov equivalence classes; see  Proposition~\ref{pro:RefinedMECcounting}. Using estimates on the tower decomposition from Section~\ref{sec:DecompositionWeights}, we then partition the sites of a given \DAG according to their potential to create non-collider edges. Bounding the number of sites in each such partition, as well as the number of non-collider edges between them, this allows us to conclude Theorem~\ref{thm:Moments} in Section~\ref{sec:NumberMECs}.

\section{Tower decomposition of \DAGs}\label{sec:HeightDecomposition}

In the following, we provide a natural way of forming equivalence classes for labeled \DAGsc apart from the Markov equivalence relation given in the introduction. We partition directed acyclic graphs according to their vertex heights, i.e.\ the length of the longest directed path to a sink in the \DAG starting from a given site $v$. A version of this construction can be found for example in Section 4 of \cite{M:ShapeDAG}. \\

For a given \DAG $G$, we define its \textbf{tower decomposition} as follows. We let $H_1=H_1(G)$ be the set of all sinks in $G$, i.e.\ the set of sites with no outgoing edges. In a next step, given $H_1$, we consider the reduced graph $G_1:= G \setminus H_1$, where we remove all sites in $H_1$ from $G$, as well as the corresponding edges to a site in $H_1$. In the reduced graph $G_1$, let $H_2$ be the set of all sinks, and set $G_2:=G_1 \setminus H_2$. Continuing in the same way, we recursively obtain a partition of $V(G)$ into disjoint sets $H_i$ for $i=1,\dots,\hei$ with some $\hei=\hei(G)$; see Figure~\ref{fig:PosetHeight}. We call $\hei(G)$ the \textbf{height} of the graph $G$, and refer to $(H_i)_{i \in [\hei(G)]}$ as the \textbf{layers} of $G$.
Using the tower decomposition of $G$, we obtain a vector $H=(H_1,H_2,\dots,H_{\hei(G)})$ and set
\begin{equation}
h=h(G)=(h_{1},h_2,\dots,h_{\hei(G)}) := (|H_1|,|H_2|,\dots,|H_{\hei(G)}|) \, .
\end{equation} 
When keeping only the edges between sets $H_i$ and $H_{i+1}$, we call the resulting graph the \textbf{tower} $\sigma(G)$ of $G$, and we refer to $h=h(G)$ as the corresponding \textbf{tower vector}. With a slight abuse of notation, we write $h(\sigma)$ for the tower vector of a \DAG with tower $\sigma$, 
and set 
\begin{equation}
 |h|=|\sigma| = \sum_{i=1}^{\hei(G)} h_i
\end{equation} as the \textbf{size} of $G$. Note that the uniform measure $\P_n$ is supported on the set $\mathcal{D}_n$ of \DAGs of size $n$.  The following is a simple observation on this construction, so we omit the proof.

\begin{lemma}\label{lem:UniquenessDecomposition} Each \DAG $G$ belongs to a unique tower $\sigma(G)$, respectively a unique tower vector $h(G)$. For each $n\in \N$, the set of towers $\Sigma_n$, and the set of tower vectors $\Omega_n$, partition the set of all \DAGs $\mathcal{D}_n$ on $n$ sites. 
\end{lemma}

We are interested in counting of the number of DAGs associated to a given tower~$\sigma$. For each tower $\sigma$, let $\mathcal{D}(\sigma)$ be the set of all \DAGs associated to $\sigma$. For each site $v\in V$, we let $d(v)$ denote the number of sites in $\sigma$, which are contained in layers above $v$, but not in the layer adjacent to $v$.

\begin{lemma}\label{lem:SkeletonCounting} For each labeled tower $\sigma$, we have that
\begin{equation}
| \mathcal{D}(\sigma)| = \prod_{v\in V(\sigma)} 2^{d(v)} \, .
\end{equation}
\end{lemma}
\begin{proof} Recall that we consider the set of labeled \DAGs on a given number of sites, and thus all sites are distinguishable. Fix some tower $\sigma$. Then each edge in the skeleton which connects two sites in adjacent layers is present for all graphs in $\mathcal{D}(\sigma)$. Between any two sites $v$ and $w$ in different layers which are not adjacent, adding or removing the edge between $v$ and $w$ does not change the underlying tower. This justifies the above formula.
\end{proof}

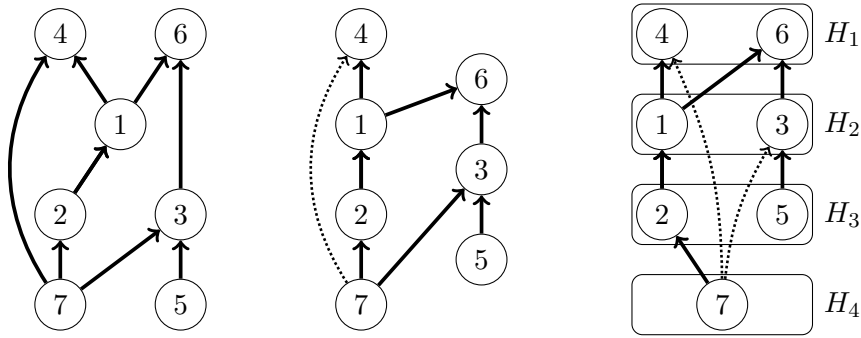
\begin{figure} 
\begin{center}
\begin{tikzpicture}[scale=0.38]

	\node[shape=circle,scale=1,draw] (A) at (-2,-1){$7$} ;
 	\node[shape=circle,scale=1,draw] (B) at (2,-1){$5$} ;
	\node[shape=circle,scale=1,draw] (C) at (-2,2){$2$} ;
 	\node[shape=circle,scale=1,draw] (D) at (2,2){$3$} ;	
 	\node[shape=circle,scale=1,draw] (E) at (0,5){$1$} ;
	\node[shape=circle,scale=1,draw] (F) at (-2,8){$4$} ;
 	\node[shape=circle,scale=1,draw] (G) at (2,8){$6$} ;

 	  	\draw[line width=1.4,->] (A) to (C);		
 	  	\draw[line width=1.4,->] (A) to (D);	
 	  	\draw[line width=1.4,->] (B) to (D);		
 	  	\draw[line width=1.4,->] (C) to (E);
	  	\draw[line width=1.4,->] (D) to (G);
	  	\draw[line width=1.4,->] (C) to (E);
	  		
	  	\draw[line width=1.4,->] (E) to (F);			
	  	\draw[line width=1.4,->] (E) to (G);	
	  	\draw[line width=1.4,->,bend left=33] (A) to (F);	
 
 \def\x{10};

	\node[shape=circle,scale=1,draw] (A) at (-2+\x,-1){$7$} ;
 	\node[shape=circle,scale=1,draw] (B) at (2+\x,0.5){$5$} ;
	\node[shape=circle,scale=1,draw] (C) at (-2+\x,2){$2$} ;
 	\node[shape=circle,scale=1,draw] (D) at (2+\x,3.5){$3$} ;	
 	\node[shape=circle,scale=1,draw] (E) at (-2+\x,5){$1$} ;
	\node[shape=circle,scale=1,draw] (F) at (-2+\x,8){$4$} ;
 	\node[shape=circle,scale=1,draw] (G) at (2+\x,6.5){$6$} ;

 	  	\draw[line width=1.4,->] (A) to (C);		
 	  	\draw[line width=1.4,->] (A) to (D);	
 	  	\draw[line width=1.4,->] (B) to (D);		
 	  	\draw[line width=1.4,->] (C) to (E);
	  	\draw[line width=1.4,->] (D) to (G);
	  	\draw[line width=1.4,->] (C) to (E);
	  		
	  	\draw[line width=1.4,->] (E) to (F);			
	  	\draw[line width=1.4,->] (E) to (G);	
	  	\draw[line width=1,->,bend left=33, densely dotted] (A) to (F);

  \def\x{20};
 
		\node[shape=circle,scale=1,draw] (A) at (0+\x,-1){$7$} ;
 	\node[shape=circle,scale=1,draw] (B) at (2+\x,2){$5$} ;
	\node[shape=circle,scale=1,draw] (C) at (-2+\x,2){$2$} ;
 	\node[shape=circle,scale=1,draw] (D) at (2+\x,5){$3$} ;	
 	\node[shape=circle,scale=1,draw] (E) at (-2+\x,5){$1$} ;
	\node[shape=circle,scale=1,draw] (F) at (-2+\x,8){$4$} ;
 	\node[shape=circle,scale=1,draw] (G) at (2+\x,8){$6$} ;

 	  	\draw[line width=1.4,->] (A) to (C);		
 	  	\draw[line width=1,->, densely dotted,bend left=12] (A) to (D);	
 	  	\draw[line width=1.4,->] (B) to (D);		
 	  	\draw[line width=1.4,->] (C) to (E);
	  	\draw[line width=1.4,->] (D) to (G);
	  	\draw[line width=1.4,->] (C) to (E);
	  		
	  	\draw[line width=1.4,->] (E) to (F);			
	  	\draw[line width=1.4,->] (E) to (G);	
        \draw[line width=1,->,bend right=12,densely dotted] (A) to (F);

 	\draw[rounded corners] (-3+\x,-2) rectangle (-3+\x+6,-2+2) {};
 	\draw[rounded corners] (-3+\x,-2+3) rectangle (-3+\x+6,-2+2+3) {};
 	\draw[rounded corners] (-3+\x,-2+9) rectangle (-3+\x+6,-2+2+9) {};
 	\draw[rounded corners] (-3+\x,-2+6) rectangle (-3+\x+6,-2+2+6) {}; 	

		\node[shape=circle,scale=1] (AX) at (2+\x+2,-1){$H_4$} ;
 	\node[shape=circle,scale=1] (BX) at (2+\x+2,2){$H_3$} ;
	\node[shape=circle,scale=1] (CX) at (2+\x+2,5){$H_2$} ;
 	\node[shape=circle,scale=1] (DX) at (2+\x+2,8){$H_1$} ;

%
%
%
%

\end{tikzpicture}
\end{center}
\caption{\label{fig:PosetHeight} Example of a \DAG with its poset and tower decomposition (the poset decomposition is defined in Section \ref{sec:PosetDecomposition}). The dotted edges are not present in the respective decomposition.} 
\end{figure}
\begin{remark}\label{rem:IndependentSamplingTower} Rephrasing Lemma \ref{lem:SkeletonCounting}, the above tower decomposition has the key property that when we condition on the underlying tower $\sigma$ in $\P_n$, each edge between two sites in layers of distance at least $2$ in the tower $\sigma$ is present independently with probability $\frac{1}{2}$.
\end{remark}

\section{Computation of the tower weights}\label{sec:DecompositionWeights}

In the following, our goal is to estimate the size of the tower for a \DAG on $n$ sites, drawn according to $\P_n$ for some $n\in \N$. We define the \textbf{weight} $w(\sigma)$ of a tower $\sigma$~as
\begin{equation}
w(\sigma) := \frac{|\mathcal{D}(\sigma)|}{Z^{(n)}} \quad \text{ with } \quad Z^{(n)}:=\sum_{\sigma \in \Sigma_n} |\mathcal{D}(\sigma)| \, .
\end{equation}  Similarly, with a slight abuse of notation, let $w(h^{\prime})$ be the \textbf{weight} of a tower vector $h^{\prime}$ with
\begin{equation}
w(h^{\prime}) := \sum_{ \sigma \in \Sigma_n \colon h(\sigma)=h^{\prime}} w(\sigma) \, ,
\end{equation} counting the number of \DAGs associated to a given tower vector suitably renormalized. The next characterization is a consequence of Lemma \ref{lem:SkeletonCounting} together with Remark~\ref{rem:IndependentSamplingTower}, and can be found for example in the proof of Theorem 2.1 in \cite{M:ShapeDAG}.
\begin{lemma}\label{lem:WeightEstimate}
For all $n,\ell\in \N$, all tower vectors $h^{\prime}=(h_1,h_2,\dots,h_\ell)$ with $|h^{\prime}|=n$
\begin{equation}
\P_{n}( h(G)= h^{\prime} ) = w(h^{\prime})  .
\end{equation} Furthermore, we have that 
\begin{equation}\label{eq:WeightEstimate}
w(h^{\prime})= \frac{1}{|\mathcal{D}_n|} \binom{n}{h_1,\, h_2, \, \dots, \, h_{\ell}}\prod_{i\in [\ell]}  2^{ h_i\sum_{j=1}^{i-1} h_j } \left( 1- 2^{-h_{i}}\right)^{h_{i+1}} \, ,
\end{equation} where we recall that $|\mathcal{D}_n|$ is the number of \DAGs on $n$ sites, and the second factor denotes the multinomial in the coefficients of $h$.
\end{lemma}

Note that the number of \DAGs corresponding to a given tower vector consisting of $n$ sites is maximized for a simple path $(v_1,v_2,\dots,v_n)$, where $d(v_i)=n-i-1$ for all $i\in [n]$. The following is a consequence of the above computations. We use the convention that $h_i(G)=0$ if $i \geq \hei(G)$.

\begin{proposition}\label{pro:LevelEstimate} For any $i\in \N$, for all $x \geq 5$, and all $n\in \N$ sufficiently large, 
\begin{equation}\label{eq:MultipleHeightPointsSingle}
\P_{n}( h_{i} \geq  x ) \leq 2^{- x^2/4 + 2} .
\end{equation} 
In particular, for $i_1,i_2,\dots,i_{k}\in \N$, for all $x_j \geq 5$ with $j\in [k]$, and all $n\in \N$ large enough,
\begin{equation}\label{eq:MultipleHeightPoints}
\P_{n}( \text{there exists some } k\in [j] \text{ such that } h_{i_k} \geq  x_{k} ) \leq 4 \sum_{k=1}^{j}2^{- x_k^2/4} .
\end{equation} 
\end{proposition}
\begin{proof} In the following, we identify each fixed height vector $h=(h_1,\dots,h_\ell)$ with $h_i=x$ for some fixed $x\in \N$ with a unique height vector $\tilde{h}$.  More precisely, for fixed $i\in [\ell]$, we consider the set of \DAGsc with respect to the vector
\begin{equation}\label{def:VectorExtension}
\tilde{h} := (h_1,h_2,\dots,h_{i-1},1,1,\dots,1,h_{i+1},h_{i+2},\dots,h_{\ell})
\end{equation} with $|h|=|\tilde{h}|$, i.e.\ we replace $h_i$ by $|h_i|$ many ones, and shift the remaining components.  We claim that for all $x\geq 3$, this new vector satisfies
\begin{equation}\label{eq:NewHeightVector}
\frac{w(h)}{w(\tilde{h})} \leq 2^{- x^2/4} \, ,
\end{equation} provided that $h_{i+1}\leq x^2/4$ holds. To see this,  by Lemma~\ref{lem:WeightEstimate}, whenever $h_{i+1}\leq x^2/4$
\begin{equation}\label{eq:WeightRatio}
\frac{w(h)}{w(\tilde{h})} \leq \frac{1}{x!} 2^{- \binom{x-1}{2} + h_{i+1}} \leq  2^{- x^2/4} 
\end{equation} holds for all $x\geq 3$. Thus, we obtain that
\begin{align}\label{eq:RecursiveheightVector}
\P_n(h_i \geq x ) &\leq \P_n\Big(h_{i+1}\geq \frac{x^2}{4}\Big) + \sum_{y\geq x} \P\Big(h_i = y \, \Big| \,  h_{i+1}\leq \frac{x^2}{4}\Big) \\ &\leq
\P\Big(h_{i+1} \geq \frac{x^2}{4}\Big) + 2 \cdot 2^{- x^2/4} \, , 
\end{align} recalling the convention that $h_{i+1}=0$ for all $i> \ell$. Using \eqref{eq:RecursiveheightVector}, we proceed by induction over $x \geq 5$, and show that 
\begin{equation}\label{eq:SingleHeightPoint}
\P_{n}(h_{i} \geq x ) \leq 2 (1+2^{-x^2/8}) 2^{- x^2/4} \, .
\end{equation} For all $x \geq n$, the inequality  \eqref{eq:SingleHeightPoint} trivially holds. Suppose that \eqref{eq:SingleHeightPoint} holds for all $y \in \{x+1,x+2,\dots\}$. Then we see by \eqref{eq:RecursiveheightVector} that
\begin{align*}
\P_n(h_i \geq x ) \leq 2 \cdot 2^{- x^2/4} + (1+2^{-x^4/128}) 2^{-x^4/64} \leq 2 (1+2^{-x^2/8}) 2^{- x^2/4}
\end{align*}
for all $x\geq 5$, which gives \eqref{eq:MultipleHeightPointsSingle}. This yields \eqref{eq:MultipleHeightPoints} by a union bound.
\end{proof}

Next, we bound the probability of having a single element in a given layer.

\begin{lemma}\label{lem:SingleGenerationHeightVector}
For any $i\in \N$ and for all $n\in \N$ sufficiently large, we have with $\P_n$-probability at least $\frac{1}{128}$ that $h_i \leq 1$ holds.
\end{lemma}
\begin{proof} Without loss of generality, assume that the event $\{\hei(G) \geq i \}$ holds with $\P_n$-probability at least $\frac{1}{2}$, as the claim is trivial, otherwise. 
Fix some $G$ drawn according to $\P_n$. When $i= \hei(G)$, the claim follows from \cite{L:MaximalVertices}, where the law of the number of sources and sinks in a \DAG is studied. For $i < \hei(G)$, observe that by Proposition \ref{pro:LevelEstimate}, we have that with $\P_n$-probability at least $\frac{1}{2}$, it holds that $\max(h_i,h_{i+1})\leq 4$ for all $n$ sufficiently large. Now let $X$ be the set of tower vectors $h\in \Omega_n$ such that $\max(h_i,h_{i+1})\leq 4$ and note that
\begin{equation}\label{eq:TowerVectorExtended}
\sum_{ h\in X} w(h) \geq \frac{1}{2} .
\end{equation} Next, let $M$ be the mapping that takes a tower vector $h$ and maps it to the tower vector $\tilde{h}$ as in \eqref{def:VectorExtension}, i.e.\ we replace the layer $H_i$ by $h_i$ many consecutive ones. Observe that the map $M$ is at most $4$ to $1$ on the domain $X$ by construction. Moreover, the relation  \eqref{eq:WeightRatio} in Proposition \ref{pro:LevelEstimate} ensures that for any $h \in X$
\begin{equation}
w\big(\tilde{h}\big) \geq \frac{1}{16} w(h) 
\end{equation} where $\tilde{h}=M(h)$.  Combining the above observations, we see that
\begin{align}
\P_n( h_i \leq 1 ) \geq \sum_{ \tilde{h}\in M(X)} w(\tilde{h}) \geq \frac{1}{4} \left( \min_{h \in X} \frac{w(M(h))}{w(h)} \right) \sum_{h \in X} w(h) \geq \frac{1}{4} \cdot \frac{1}{16}\cdot \frac{1}{2}  =  \frac{1}{128} .
\end{align} This finishes the proof.
\end{proof}

\begin{remark}\label{rem:AllGenerations}
The same arguments as in Lemma \ref{lem:SingleGenerationHeightVector} show that for any $k\geq 1$, there is a strictly positive probability $p_k>0$ to see exactly $k$ sites in a fixed layer $i$, provided that $n$ is sufficiently large.
\end{remark}

We remark that previously exact asymptotic bounds -- in terms of Gaussian tails -- were only known for the last layer using generating functions; see \cite{L:MaximalVertices}. Proposition  \ref{pro:LevelEstimate} yields Gaussian tails for any generation. Next, we consider the vertical distance until we reach a layer containing only a single element. We say that $H_j$ is a \textbf{regeneration layer} if and only if $h_j=1$. Moreover, we let $\tau_i$ be the $i^{\textup{th}}$ \textbf{regeneration point}, i.e.\ we set
\begin{equation}
\tau_1 := \min\{j\in [n]  \colon h_j=1\} \, ,
\end{equation}
and recursively 
\begin{equation}
\tau_{i+1} := \min \{ j > \tau_{i+1} \colon |h_j| = 1 \} \, ,
\end{equation} whenever there are at least $i+1$ many regeneration points in the tower for some $i\in \N$. 
Recall that $\Omega_n$ denotes the space of tower vectors of size $n$. In the following, it will be convenient to write $\Omega^{r}$ for the space of tower vectors with $r$ regeneration points, and we set
\begin{equation}\label{def:SpaceOfSizeAndRegeneration}
\Omega_{n}^{r} := \Omega_n \cap \Omega^{r}
\end{equation} as the space of tower vectors of size $n$ with exactly $r$ regeneration points. Moreover, we set \begin{equation}\label{def:SpaceOfAllHeightTowers}
\Omega:=\bigcup_{n\in \N} \Omega_n = \bigcup_{r \geq 0} \Omega^r
\end{equation} as the space of all tower vectors. The following result, which is a consequence of Theorem~3.3 in~\cite{M:ShapeDAG}, gives a bound on the total number of regeneration points $R(h)$ of a tower vector $h$.
\begin{lemma}\label{lem:NumberOfLevels}
There exist constants $c_1,c_2>0$ such that for all $s \in \R$
\begin{equation}
\lim_{n \rightarrow \infty} \P_n\left( \frac{R(h)-c_1n}{c_2 \sqrt{n}} \geq s \right) = \P(X \geq s) \, ,
\end{equation} where $X$ is a standard normal random variable.
\end{lemma}
We use the conventions that $\tau_0=0$, and $\tau_i=\hei(G)+1$ if $h$ contains less than $i$ regeneration points. Let $A_n$ be the event that the tower of a \DAG according to $\P_n$ has at least $\frac{1}{2}c_1 n$ many regeneration points, and note that by Lemma~\ref{lem:NumberOfLevels}
\begin{equation}\label{eq:NumberRegens}
\lim_{n \rightarrow \infty}\P_{n}(A_n) = 1 \, .
\end{equation}
We conclude this section with the following consequence of the above results on the distance between regeneration points, which is of independent interest.
\begin{corollary}\label{cor:DistanceOfLevels}
There exists some universal constant $c>0$ such that for all $n\in \N$ sufficiently large, and all $s \in \N$, we have that
\begin{equation}\label{eq:NextRegenPoint1}
\P_{n}(\tau_1 \geq s ) \leq \exp(-cs)
\end{equation} as well as for all $i \in \N$
\begin{equation}\label{eq:NextRegenPoint2}
\P_{n}(\tau_i-\tau_{i-1} \geq s ) \leq \exp(-cs) \, .
\end{equation} 
\end{corollary}
\begin{proof} Fix some $n\in \N$ large enough. We claim that it suffices to show that for any fixed sequence of integers $(x_1,x_2,\dots,x_i)$ with $\sum_{k\in [i]} x_k < n$, we have
\begin{equation}\label{eq:ConditionedPolymer}
\P_{n}( h_{j+1} \leq 1 \,  | \, h_k=x_k \text{ for all } k\in [j-1]) \geq \tilde{c}
\end{equation} for some universal constant $\tilde{c}>0$, uniformly in $j$ and $n$. The bounds in \eqref{eq:NextRegenPoint1} and \eqref{eq:NextRegenPoint2} then follow from \eqref{eq:ConditionedPolymer} by revealing the entries of  the tower vector sequentially. To see that \eqref{eq:ConditionedPolymer} holds, notice from the construction of the weights in \eqref{eq:WeightEstimate} that the number of sites in layer $j$ depends only on the total number of sites, which were not assigned to the first $j-1$ layers. Thus, by Lemma \ref{lem:SingleGenerationHeightVector}, we get that layer $j$ contains with probability $\tilde{c}>0$ at most one element, uniformly in $n$ and $j$. This gives  \eqref{eq:ConditionedPolymer}.
\end{proof}

\section{A Decomposition of tower vectors}\label{sec:DecompositionOfTowers}

In the following, our goal is to use the tower decomposition in order to study  the structure of non-collider edges. Fix some tower vector $h$, and recall that $R(h)$ is the number of regeneration points in $h$. We denote by $(\mathcal{G}_i)_{i \in [R(h)] \cup \{0\}}$ the collection of layers  between regeneration points $\tau_{i}$ and $\tau_{i+1}$, with the convention that $\tau_0=0$. We call $\mathcal{G}_i$ the $i^{\textup{th}}$ \textbf{feature} in $h$; see Figure \ref{fig:FeatureDecomposition}. Whenever $\mathcal{G}_i$ is non-empty, we write
\begin{equation}
\mathcal{G}_i = (g^{i}_1,g^{i}_2,\dots,g^{i}_{\tau_{i+1}-\tau_{i}-1})
\end{equation} for integers $g_1^{i},g_2^{i},\dots, g^{i}_{\tau_{i+1}-\tau_{i}-1} \geq 2$. Further, let $\mathcal{L}(\mathcal{G}_i):=\tau_{i+1}-\tau_{i}-1$ be the \textbf{length}, and let $|\mathcal{G}_i|:= \sum^{\mathcal{L}(\mathcal{G}_i)}_{j=1} g^{i}_{j}$ be the \textbf{size} of $\mathcal{G}_i$. In the following, we change our perspective and construct a tower vector out of a family of i.i.d.\ features. With a slight abuse of notation, for a given feature $\mathcal{G}=(g_1,\dots,g_m)$, we let
\begin{equation}\label{eq:FeatureWeights}
w(\mathcal{G}) := 2^{-\binom{|\mathcal{G}|}{2}-|\mathcal{G}|}\prod_{i=1}^{m+1} \frac{1}{g_i!}\Big( 2^{ g_i\sum^{i-1}_{j=0}g_j}\Big)  (1-2^{-g_{i-1}})^{g_i}
\end{equation} denote the \textbf{weight} of $\mathcal{G}$, with the convention that $g_0=g_{m+1}=1$. Let $\Gamma$ denote the set of all features, including also the empty set $\emptyset$ with weight $w(\emptyset):=1$, and $|\mathcal{G}_i|=0$ if $i> R(h)$.
Intuitively, we obtain the above weight of a feature $\mathcal{G}$ by comparing the weight of the sequence $(g_1,g_2,\dots,g_m)$ inside a tower vector between two regeneration points to the weight of a tower vector when replacing the corresponding part $\mathcal{G}$ by a path of length $|\mathcal{G}|$. \\

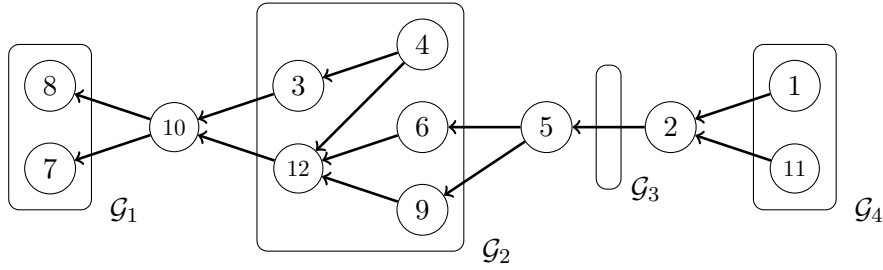
\begin{figure}
\begin{center}
\begin{tikzpicture}[scale=0.55]

	\node[shape=circle,scale=1,draw] (A) at (-6,-1){$7$} ;
 	\node[shape=circle,scale=1,draw] (B) at (-6,1){$8$} ;
	\node[shape=circle,scale=0.8,draw] (C) at (-3,0){$10$} ;
 	\node[shape=circle,scale=1,draw] (D) at (0,1){$3$} ;	
 	\node[shape=circle,scale=0.8,draw] (E) at (0,-1){$12$} ;
	\node[shape=circle,scale=1,draw] (F) at (3,2){$4$} ;
 	\node[shape=circle,scale=1,draw] (G) at (3,0){$6$} ; 
	\node[shape=circle,scale=1,draw] (H) at (3,-2){$9$} ;
 	\node[shape=circle,scale=1,draw] (I) at (6,0){$5$} ;
	\node[shape=circle,scale=1,draw] (J) at (9,0){$2$} ;
 	\node[shape=circle,scale=1,draw] (K) at (12,1){$1$} ;	 	
 	\node[shape=circle,scale=0.8,draw] (L) at (12,-1){$11$} ;
 	 	  	
 	  	\draw[line width=1,->] (C) to (A);		
 	  	\draw[line width=1,->] (C) to (B);	
 	  	\draw[line width=1,->] (D) to (C);		
 	  	\draw[line width=1,->] (E) to (C);

 	  	\draw[line width=1,->] (F) to (D);		
 	  	\draw[line width=1,->] (F) to (E);	
 	  	\draw[line width=1,->] (G) to (E);		
 	  	\draw[line width=1,->] (H) to (E);

 	  	\draw[line width=1,->] (I) to (G);		
 	  	\draw[line width=1,->] (I) to (H);
 	  	\draw[line width=1,->] (J) to (I);	  		

 	  	\draw[line width=1,->] (K) to (J);
 	  	\draw[line width=1,->] (L) to (J);	 
 	  	
\draw[rounded corners] (-7,-2) rectangle (-5,2) {}; 	 

\node[shape=circle,scale=1] (AX) at (-4.2,-2){$\mathcal{G}_1$} ;
 	  	
\draw[rounded corners] (-1,-3) rectangle (4,3) {}; 

\node[shape=circle,scale=1] (AX) at (4.8,-3){$\mathcal{G}_2$} ;

\draw[rounded corners] (7.2,-1.5) rectangle (7.8,1.5) {}; 	

\node[shape=circle,scale=1] (AX) at (8.4,-1.5){$\mathcal{G}_3$} ;

\draw[rounded corners] (11,-2) rectangle (13,2) {}; 	  

\node[shape=circle,scale=1] (AX) at (13.8,-2){$\mathcal{G}_4$} ;
 \end{tikzpicture}
\end{center}
\caption{\label{fig:FeatureDecomposition} Example of the decomposition of a tower into four features with $\mathcal{G}_0=(2)$, $\mathcal{G}_1=(2,3)$, $\mathcal{G}_2=\emptyset$, and $\mathcal{G}_3=(2)$. } 
\end{figure}

We show in Lemma \ref{lem:GadgetEquivalence} that the weight of a feature is   compatible with the weight of a tower vector defined in \eqref{eq:WeightEstimate}.
However, the above weight in \eqref{eq:FeatureWeights} is adjusted for wired boundaries, i.e.\ we are using the convention $g_0=g_{m+1}=1$. We will require slightly different weights to model the free boundaries in the features to the left of the first regeneration point and to the right of the last regeneration point. More precisely, we denote by 
\begin{equation}\label{def:FreeLeft}
\overleftarrow{w}(\mathcal{G}) := 2^{-\binom{|\mathcal{G}|}{2}}\prod_{i=2}^{m+1} \frac{1}{g_i!}\Big( 2^{ g_i\sum^{i-1}_{j=1} g_j}\Big) (1-2^{-g_{i-1}})^{g_i}
\end{equation} the weight of a feature $\mathcal{G}$ with \textbf{free left boundary} as well as by
\begin{equation}\label{def:FreeRight}
\overrightarrow{w}(\mathcal{G}) := 2^{-\binom{|\mathcal{G}|}{2}}\prod_{i=1}^{m} \frac{1}{g_i!}\Big( 2^{ g_i\sum^{i-1}_{j=0}g_j}\Big)  (1-2^{-g_{i-1}})^{g_i}
\end{equation} with weight of a feature $\mathcal{G}$ with \textbf{free right boundary}. Again, in both of the above cases, we use the convention that the  weight is set to $1$ whenever $\mathcal{G}$ is the empty feature. Moreover, for a feature $\mathcal{G}$ and some $\theta \in \R$, we set
\begin{equation}
w_{\theta}(\mathcal{G}) := \exp(- \theta |\mathcal{G}|) w(\mathcal{G}) \, ,
\end{equation} and similarly for $\overleftarrow{w_{\theta}}(\mathcal{G})$ and $\overrightarrow{w_{\theta}}(\mathcal{G})$, as its \textbf{tilted weight}. 
We obtain the following estimate on the tilted weights of all possible features.
\begin{lemma}\label{lem:PolymerInterpretation}
There exists some $\theta_0 < 0$ such that for all $\theta>\theta_0$, we have that
\begin{equation}\label{eq:Polymer1}
Z_{\theta}:=  \sum_{\mathcal{G} \in \Gamma} w_{\theta}(\mathcal{G}) < \infty \, .
\end{equation} Moreover, for all $\theta>\theta_0$, there exists some $\varepsilon_0=\varepsilon_0(\theta)>0$ such that for all $\varepsilon< \varepsilon_0$
\begin{equation}\label{eq:Polymer2}
Z_{\theta}^{-1}  \sum_{\mathcal{G} \in \Gamma}  \exp(\varepsilon|\mathcal{G}|)  w_{\theta}(\mathcal{G}) < \infty \, .
\end{equation} 
\end{lemma}
\begin{proof} 
Using that $Z_{\theta}$ is monotone decreasing in $\theta$, it suffices to show that \eqref{eq:Polymer1} holds for some fixed $\theta < 0$. Using the facts that
\begin{equation}
\frac{1}{2}(|\mathcal{G}|+1)^2 = \frac{1}{2}\left( \sum_{i=1}^{m+1} g_i^2 \right) +  \left( \sum_{i=1}^{m+1} g_i \sum_{j=1}^{i-1} g_j  \right)
\end{equation} and that $g_i \geq 2$ for all $i\in [m]$ whenever $\mathcal{G}$ is non-empty, we get  for all $\mathcal{G} \neq \emptyset$ that
\begin{equation}
w(\mathcal{G}) \leq 8 \prod_{i=1}^{m} \frac{1}{g_i!}2^{\frac{3}{2}g_i- \frac{1}{2}g_i^2 }(1-2^{-g_i})^2 \, .
\end{equation} Summing now over all $\mathcal{G}$ with $|\mathcal{G}|=k$ for some fixed $k \geq 2$, we get that
\begin{equation}\label{eq:ProductEstimate}
\sum_{ \mathcal{G} \colon |\mathcal{G}|=k} w(\mathcal{G}) \leq 8 \Big(\sum_{x \geq 2} \frac{1}{x!}2^{-\frac{1}{2}x^2 +\frac{3}{2}x}(1-2^{-x})^2  \Big)^k \, .
\end{equation} 
A computation now shows that the right-hand side of \eqref{eq:ProductEstimate} is less than $8 \cdot (\frac{3}{4})^k$ for all $k \geq 2$, allowing us to conclude \eqref{eq:Polymer1}.
Moreover, since we have that 
\begin{equation}
\exp(\varepsilon |\mathcal{G}|) w_{\theta} = w_{\theta-\varepsilon} , 
\end{equation}
 the statement \eqref{eq:Polymer2} follows immediately for $\varepsilon < \theta - \theta_0$. 
\end{proof}

\begin{corollary}\label{cor:ExtensionToBoundaryWeights}
For all $\theta>\theta_0$ with $\theta_0$ from Lemma \ref{lem:PolymerInterpretation}, it holds that
\begin{equation}\label{eq:FreeBoundarySums}
\sum_{\mathcal{G} \in \Gamma} |\mathcal{G}|\overleftarrow{w_{\theta}}(\mathcal{G}) < \infty  \quad \text{ and } \quad \sum_{\mathcal{G} \in \Gamma} |\mathcal{G}|\overrightarrow{w_{\theta}}(\mathcal{G}) < \infty  .
\end{equation}
\end{corollary}
\begin{proof}
Observe that we have for all features $\mathcal{G}=(g_1,\dots,g_j)\in\Gamma$ with $j\in \N$ 
\begin{equation}
\overleftarrow{w_{\theta}}(\mathcal{G}) \leq 2^{g_1} w_{\theta}(\mathcal{G})  .
\end{equation} Since the weights $w_{\theta}(\mathcal{G})$ are super-exponentially decaying in the size of the first component $g_1$, this bounds the first sum in \eqref{eq:FreeBoundarySums}. A similar argument holds with respect to $\overrightarrow{w_{\theta}}(\mathcal{G})$.
\end{proof}

The above construction of features has the advantage that we can define for all $\theta \geq 0$ the law $\mu_{\theta}$ on $\Gamma$ by
\begin{equation}\label{def:Mumeasure}
\mu_{\theta}(\mathcal{G}) := Z^{-1}_{\theta} w_{\theta}(\mathcal{G})  ,
\end{equation} and similarly the laws $\overleftarrow{\mu_\theta}$ and $\overrightarrow{\mu_\theta}$. 
We compose now the above features in order to describe a corresponding tower vector. More precisely, we define a tower vector $\mathcal{X}$ with a given number of regeneration points $r=R(\mathcal{X})$ using a collection of (potentially empty) features $\mathcal{G}_0,\mathcal{G}_1,\dots,\mathcal{G}_{r-1},\mathcal{G}_r$ according to $\mu_{\theta}$ for $\mathcal{G}_i$ with $i\in [r-1]$, according to $\overleftarrow{\mu_\theta}$ for $\mathcal{G}_0$, and according to $\overrightarrow{\mu_\theta}$ for $\mathcal{G}_r$. We denote by $\Q^{r}_{\theta}$ the corresponding law of a tower vector $\mathcal{X}$ on the space $\Omega^{r}$ under this construction. 
We have the following key result on the sampling of features for a given total size and a given number of regeneration points.

 \begin{lemma}\label{lem:GadgetEquivalence}
Fix some tower vector $\tilde{h}$ with $r$ regeneration points and size $n$, i.e.\ the corresponding tower vector belongs to $\Omega_{n}^{r}$. Then for any $\theta$ such that $\Q^{r}_{\theta}$ is well-defined
\begin{equation}
\P_n( h = \tilde{h}\,  | \, R(h)=r) = \Q^{r}_{\theta}\big( \mathcal{X}=\tilde{h} \, \big| \,  |\mathcal{X}|=n  \big) \, .
\end{equation} 
\end{lemma}
\begin{proof}
Observe that the measure
\begin{equation}
\Q^{r}_{\theta}\big( \mathcal{X} \in \cdot \, \big| \,  |\mathcal{X}|=n  \big) 
\end{equation} is supported on $\Omega_{n}^{r}$ and does not depend on $\theta$. Fix some $\theta> \theta_0$ with $\theta_0$ from Lemma \ref{lem:PolymerInterpretation}. The probability to see $\mathcal{X}$ according to $\Q^{r}_{\theta}$ with features $(\mathcal{G}_i)_{i\in \{0\} \cup[r]}$ is proportional to 
\begin{equation}
W_r(\mathcal{X}) := \overleftarrow{w_0}(\mathcal{G}_0) \left( \prod_{i=1}^{r-1} w_{0}(\mathcal{G}_i) \right) \overrightarrow{w_0}(\mathcal{G}_r) \, .
\end{equation} 
Recall the weight $w(\mathcal{X})$ from \eqref{eq:WeightEstimate} for a tower vector $\mathcal{X} \in \Omega^{r}_n$. Let $h_{\mathds{1}}^{n}$ be the all one vector of length $n$, and note that the vector $h_{\mathds{1}}^{n}$ and its weight $w(h_{\mathds{1}}^{n})$ do not depend on $r$. We claim that on the space $\Omega_n$ for some fixed $n\in \N$
\begin{equation}\label{eq:TowerIdentity}
W_{r}(\mathcal{X}) = \frac{w(\mathcal{X})}{w(h_{\mathds{1}}^{n})} \, .
\end{equation}
Indeed, note that for $\mathcal{X}=h_{\mathds{1}}^{n}$, the claim \eqref{eq:TowerIdentity} follows from the definition of $W_{r}$ as a product over the weight of empty features. It thus remains to argue that $W_{r}(\mathcal{X})$ is proportional to $w(\mathcal{X})$ for all $\mathcal{X} \in \Omega_{n}^{r}$. This follows by comparing the feature weights in \eqref{eq:FeatureWeights}, \eqref{def:FreeLeft}, and \eqref{def:FreeRight} to the number of \DAGs for a given tower in \eqref{eq:WeightEstimate}. Hence, we get from \eqref{eq:TowerIdentity} that
 \begin{equation}
\P_n\big( \mathcal{X} \in A \, \big| \,  R(\mathcal{X})=r  \big) = \frac{1}{\tilde{Z}_{r,n}}\sum_{\sigma \in A} W_{r}(\sigma) 
\end{equation} for some suitable normalizing constant $\tilde{Z}_{r,n}$, and all subsets $A \subseteq \Omega_{n}^{r}$. Since $\P_n\big( \cdot \,  | \,  R(\mathcal{X})=r  \big)$ is supported on $\Omega_{n}^{r}$, this finishes the proof.
\end{proof}

In the following, we use this change of measure together with Lemma \ref{lem:PolymerInterpretation} in order to analyze the tower decomposition. 
We will use the following lemma which allows us to apply Lemma \ref{lem:PolymerInterpretation} for some suitable choice of $\theta>0$. We denote by $\tilde{\mathbb{Q}}_0^{r}$ the measure where we sample a vector $\tilde{\mathcal{X}}=(\mathcal{X}_i)_{i \in [r]}$ on the space $\Gamma^r$ of $r$ features, each of them independently according to the measure $\mu_0$. 

\begin{lemma}\label{lem:FixTheta} Recall the constant $c_1>0$ from Lemma \ref{lem:NumberOfLevels}. The function 
\begin{equation}\label{eq:StrictIncrease}
\theta \mapsto \sum_{\mathcal{G} \in \Gamma} |\mathcal{G}| \mu_{\theta}(\mathcal{G})
\end{equation}
is continuous and strictly decreasing for all $\theta \geq 0$. Moreover, there exists some $\theta_{\ast}>0$ with
\begin{equation}\label{eq:ThetaStar}
\sum_{\mathcal{G} \in \Gamma} |\mathcal{G}| \mu_{\theta_{\ast}}(\mathcal{G}) = \frac{1-c_1}{c_1}  .
\end{equation} 
\end{lemma}
\begin{proof} The first part \eqref{eq:StrictIncrease} is immediate from Lemma \ref{lem:PolymerInterpretation}. In the following, we let $d_1$ be defined such that 
\begin{equation}\label{eq:Theta0Case}
\sum_{\mathcal{G} \in \Gamma} |\mathcal{G}| \mu_{0}(\mathcal{G}) = \frac{1-d_1}{d_1} \, . 
\end{equation} We will argue in the following that $d_1<c_1$. Since the right-hand side in \eqref{eq:StrictIncrease} goes to $0$ for $\theta \rightarrow \infty$, this yields \eqref{eq:ThetaStar} using the continuity and monotonicity in \eqref{eq:StrictIncrease}. Using the notation $|\tilde{\mathcal{X}}| := \sum_{i=1}^{r} |\mathcal{X}_i|$, observe that
\begin{equation}\label{eq:TranslateToPartition}
\sum_{r \geq 1} \sum_{\tilde{\mathcal{X}} \in \Gamma^{r}} w(\tilde{\mathcal{X}}) \mathds{1}_{\{|\tilde{\mathcal{X}}|=n-r\}} \ = \sum_{r \geq 1} Z_0^{r} \tilde{\mathbb{Q}}^{r}_0(|\tilde{\mathcal{X}}|=n-r) 
\end{equation} with $Z_0$ from \eqref{eq:Polymer1}. Intuitively, the expressions in \eqref{eq:TranslateToPartition} correspond to the expected weight of all features. In order to show that $d_1<c_1$, we claim that it suffices to argue that for some $\varepsilon>0$ 
\begin{equation}\label{eq:targetEquation}
\sum_{r \leq (1+\varepsilon)d_1n} Z_0^{r} \tilde{\mathbb{Q}}^{r}_0(|\tilde{\mathcal{X}}|=n-r) \leq \exp(-cn ) \sum_{r \geq 1} Z_0^{r} \tilde{\mathbb{Q}}^{r}_0(|\tilde{\mathcal{X}}|=n-r)
\end{equation} holds with some constant $c=c(\varepsilon)>0$, and all $n$ sufficiently large. To see that \eqref{eq:targetEquation} implies $d_1<c_1$, we argue with a proof by contradiction. Assume that $d_1 = c_1$ as the case $d_1 > c_1$ is similar. Then the expression \eqref{eq:TranslateToPartition} yields that with probability tending to $1$ as $n$ goes to infinity, we must have that $\tilde{\mathcal{X}} \in \Omega_n$, drawn proportionally to the weights $w(\tilde{\mathcal{X}})$, has at least $(1+\varepsilon)c_1n$ many features. However, this is a contradiction to Lemma \ref{lem:NumberOfLevels}. \\

In order to show that \eqref{eq:targetEquation} indeed holds, we distinguish two cases. First, we consider all values of $r\leq (1-\varepsilon)d_1 n$. Then since $Z_0>1$, using a large deviation principle for the i.i.d.\ random variables $(|\mathcal{X}_i|)_{i \in [r]}$, recalling that the random variables have a finite exponential moment by Lemma \ref{lem:PolymerInterpretation}, we see that
\begin{align}\label{eq:Moderate1}
\sum_{r \leq (1-\varepsilon)d_1n} Z_0^{r} \tilde{\mathbb{Q}}^{r}_0(|\tilde{\mathcal{X}}|=n-r) \leq n \exp(-C_1n) Z^{\lfloor d_1n(1-\varepsilon/2) \rfloor}_{0} .
\end{align} for some constant $C_1>0$, and all $n$ sufficiently large.
Note that the random variables $(|\mathcal{X}_i|)_{i \in [r]}$ with mean $(1-d_1)/d_1$ have full support on $\N$ by Remark \ref{rem:AllGenerations}, and thus satisfy a local central limit theorem; see Chapter 2 in \cite{LL:RWIntroduction} for a more detailed discussion. This allows us to bound the probability of $\tilde{\mathbb{Q}}^{r}_0(|\tilde{\mathcal{X}}|=n-r)$ from below to see that
\begin{align}\label{eq:Moderate2}
\exp(-C_1n) Z^{\lfloor d_1n(1-\varepsilon/2) \rfloor}_{0} \leq \frac{1}{n} \exp(-C_2n) \sum_{r \geq \lfloor d_1n(1-\varepsilon/2) \rfloor} Z^{r}_{0} \tilde{\mathbb{Q}}^{r}_0(|\tilde{\mathcal{X}}|=n-r)
\end{align} for some constant $C_2>0$, and all $n$ sufficiently large. Combining now \eqref{eq:Moderate1} and \eqref{eq:Moderate2}, we conclude that
\begin{align}\label{eq:ModerateAll}
\sum_{r \leq (1-\varepsilon)d_1n} Z_0^{r} \tilde{\mathbb{Q}}^{r}_0(|\tilde{\mathcal{X}}|=n-r) \leq  \exp(-C_2n) \sum_{r \geq \lfloor d_1n(1-\varepsilon/2) \rfloor} Z^{r}_{0} \tilde{\mathbb{Q}}^{r}_0(|\tilde{\mathcal{X}}|=n-r) . 
\end{align}
Next, using again that the random variables $(|\mathcal{X}_i|)_{i \in [r]}$ have some finite exponential moment by Lemma \ref{lem:PolymerInterpretation}, a moderate deviation estimate shows that for all $\delta>0$ sufficiently small and $n$ large enough
\begin{equation}\label{eq:ModerateDeviationApprox}
\tilde{\mathbb{Q}}_0^{\lfloor nd_1(1+\delta) \rfloor}\left(|\tilde{\mathcal{X}}| = n- \lfloor nd_1(1+\delta)\rfloor) \right) \geq  \exp(-C_3\delta^2 n)
\end{equation} 
for some constant $C_3>0$; see \cite{DA:ModerateDeviations} for a reference on moderate deviations. As $Z_0>1$, we use \eqref{eq:ModerateDeviationApprox} for  $r=\lfloor(1+\delta)d_1 n \rfloor$ in the sum on the right-hand side of \eqref{eq:targetEquation} for sufficiently small $\delta>0$, and choose $\varepsilon=\delta/2$ on the left-hand side of \eqref{eq:targetEquation} to  conclude.
\end{proof}

 We have now all tools to show the convergence of the last entries in the tower vector.

\begin{lemma}\label{lem:GadgetSize}
We have for any fixed $k\in \N$ and $s_i \geq 0$ with $i\in [k]$ that
\begin{equation}\label{eq:LimitExistsGadget}
\lim_{n \rightarrow \infty} \P_{n}( h_{\hei(G)+1- i} \geq s_{i} \text{ for all } i\in [k] ) = \P(Y_i \geq s_i \text{ for all } i\in [k] )
\end{equation} for some almost surely finite random variables $(Y_i)_{i\in [k]}$ with full support on $\N \cup \{ 0 \}$.
\end{lemma}
\begin{proof}
Note that by Proposition \ref{pro:LevelEstimate} and Remark \ref{rem:AllGenerations}, it remains to show that the above limit in \eqref{eq:LimitExistsGadget} exists. To do so, we argue that for every fixed length $k$, the law of the last $k$ entries in the tower vector converges as the total number of sites goes to infinity. We strongly rely on Lemma \ref{lem:GadgetEquivalence} in order to work with the law $\Q_{\theta}^r$ instead of $\P_{n}$, allowing for some additional flexibility in $\theta$. To formalize this, fix some $k\geq 1$ and  $\mathbf{x}=(x_0,x_1,\dots,x_k)$ with $x_j \in \Gamma$ for all such $j$, i.e.\ $\mathbf{x}$ is a family of features. We denote by $|\mathbf{x}|=|x_0|+|x_1|+\cdots+|x_{k-1}|+|x_{k}|$ the total of number of sites in these features.
Let $\mathcal{X}$ with features $(\mathcal{G}_i)_{i\in  \{0\} \cup[r]}$ be sampled according to $\Q_{\theta}^r$ on the space $\Omega^{r}$ for some fixed $r$, and for some $\theta \geq 0$. We define 
\begin{equation*}
A_{n,r}^{\mathbf{x}}:= \left\{ \mathcal{X} \in \Omega_{n}^{r} \text{ and } \mathcal{G}_0=x_0 \text{ and } \mathcal{G}_{r-k+i}=x_{i} \text{ for all } i\in [k] \text{ and }  \sum_{i=1}^{r-k} |\mathcal{G}_i| = n-|\mathbf{x}| -r \right\}
\end{equation*} as the event that the first feature, as well as the $k$ last features, agree with the collection of features $\mathbf{x}$, and that the sample $\mathcal{X}$, counting features and regeneration points, has total size $n$. Similarly, we let 
\begin{equation*}
B_{\mathbf{x}}:= \left\{ \mathcal{G}_{R(\mathcal{X})-k+i}=x_i \text{ for all } i\in [k] \text{ and } \mathcal{G}_{0}=x_0\right\}
\end{equation*} be the event that the last $k$ features in the sampled tower vector correspond to the last $k$ features in $\mathbf{x}$, and the first feature agrees with $x_0$. Note that $B_{\mathbf{x}}$ is also naturally defined as an event on the space $\Omega$, requiring at least $k$ regeneration points. Recall $\mu_\theta$ from \eqref{def:Mumeasure}. To simplify notation, we set 
\begin{equation}
\mu_{\theta}(\mathbf{x}) := \overleftarrow{\mu_{\theta}}(x_0)\left(\prod_{i=1}^{k-1}{\mu_{\theta}(x_i)}\right) \overrightarrow{\mu_{\theta}}(x_r) \, .
\end{equation} 
Next, recall from Lemma \ref{lem:NumberOfLevels} that the number of regeneration points $R(\cdot)$ according to $\P_n$ converges in the leading order to $c_1n$ for some universal constant $c_1 \in (0,1)$. 
%
%
When the event $A_n$ from \eqref{eq:NumberRegens} holds, note that in order to prove \eqref{eq:LimitExistsGadget}, 
it suffices to show that for any $\mathbf{x}=(x_0,x_1,x_2,\dots,x_k)$ with $x_i \in \Gamma$ for all such $i$, we have that
\begin{equation}\label{eq:ExpliciteConvergence}
\lim_{n\rightarrow \infty} \P_{n}(B_{\mathbf{x}}) =\mu_{\theta_\ast}(\mathbf{x}) 
\end{equation} for the constant $\theta_{\ast}$ from Lemma \ref{lem:FixTheta}. 
Fix $\varepsilon>0$ and let $I_n := [c_1 n - n^{2/3},c_1 n + n^{2/3}]$. Then for all $n$ sufficiently large, we get by Lemma \ref{lem:NumberOfLevels} that
\begin{equation}
\big| \P_n(h \in B_{\mathbf{x}}\, | \, R(h) \in I_n ) - \P_n(h  \in B_{\mathbf{x}} ) \big| \leq  \varepsilon \, .
\end{equation}
Moreover, by Lemma \ref{lem:GadgetEquivalence}, we have for any $\theta \geq 0$
\begin{align}
 \P_n( h \in B_{\mathbf{x}} \, | \, R(h) \in I_n ) &= \sum_{r\in I_n} \P_n( h\in B_{\mathbf{x}} \, | \, R(h)=r )  \P_n( R(h)=r \, | \, R(h) \in I_n ) \\
 &= \sum_{r\in I_n} \Q_{\theta}^r\big( h \in B_{\mathbf{x}} \, \big| \, |h|=n  \big)  \P_n( R(h)=r \, | \, R(h) \in I_n ) \, . \label{eq:BxDecomposition}
\end{align}
Next, we claim that for all $n\in \N$ sufficiently large, and every $r\in I_n$, there exists some $\theta=\theta(n,r)\geq 0$ such that
\begin{equation}\label{eq:ThetaConstructed}
\sum_{\mathcal{X} \in \Gamma^{r}} |\mathcal{X}|\mathbb{Q}_{\theta}^r(\mathcal{X}) = n \, .
\end{equation} 
 Moreover, it holds that
\begin{equation}\label{eq:ThetaWellEstimated}
\lim_{n \rightarrow \infty} \sup_{r \in  I_n} \Big| \theta(n,r)- \theta_{\ast} \Big| = 0 
\end{equation} for the constant $\theta_{\ast} > 0$ from \eqref{eq:ThetaStar}. The two  statements in \eqref{eq:ThetaConstructed} and \eqref{eq:ThetaWellEstimated}  follow from Lemma~\ref{lem:FixTheta} and the central limit theorem in Lemma~\ref{lem:NumberOfLevels} on the number of regeneration points,  together with Corollary~\ref{cor:ExtensionToBoundaryWeights} in order to control the first and the last feature for fixed $\theta \geq 0$. 
We claim it suffices for \eqref{eq:ExpliciteConvergence} to show that for every fixed $\mathbf{x}$
\begin{equation}\label{eq:ApproxQr}
\lim_{n \rightarrow \infty} \sup_{r \in  I_n} \big| \Q_{\theta}^r\big( h \in B_{\mathbf{x}} \, \big| \, |h|=n  \big) - \mu_{\theta_{\ast}} (\mathbf{x}) \big| = 0 \, .
\end{equation} This follows directly by inserting \eqref{eq:ApproxQr} into \eqref{eq:BxDecomposition}.
%
%
%
 We prove \eqref{eq:ApproxQr} by showing that
\begin{equation}\label{eq:FractionConvergence}
\lim_{n \rightarrow \infty } \sup_{r \in I_n}\left| \frac{\Q_{\theta}^r\big( h \in B_{\mathbf{x}} \, \big| \, |h|=n  \big)}{\Q_{\theta}^r\big( h \in B_{\mathbf{y}} \, \big| \, |h|=n  \big)} - \frac{\mu_{\theta_{\ast}}(\mathbf{x})}{\mu_{\theta_{\ast}}(\mathbf{y})} \right| = 0 
\end{equation} for any two vectors $\mathbf{x}$ and $\mathbf{y}$ of length $k$. 
Let us stress at this point that it is crucial to consider $\theta=\theta(n,r)$ in \eqref{eq:FractionConvergence}, as we will require that $\Q_{\theta}^r(  \sum_{j=1}^{r-k} \mathcal{G}_j =n -x )$ is concentrated for slightly varying values of $x$, when $n$ is sufficiently large.
In the following, to show \eqref{eq:FractionConvergence}, it is helpful to interpret the measures $\Q_{\theta}^r( \, \cdot \, | \, |h|=n  )$ and $\mu_{\theta_{\ast}}$ as probability measures on the space $\Gamma^{k+1}$, summing over all possible choices for a given $\mathbf{x} \in \Gamma^{k+1}$ for the first, and the last $k$ features.
Recall the event $A_{n,r}^{\mathbf{x}}$, and note that 
\begin{equation}
\Q_{\theta}^r\big( h \in B_{\mathbf{x}} \, \big| \, |h|=n  \big) = \Q_{\theta}^r( A_{n,r}^{\mathbf{x}} ) / \Q_{\theta}^r( |h|=n ) \, .
\end{equation}
From the definition of $A_{n,r}^{\mathbf{x}}$ and the construction of $\Q_{\theta}^r$, we see that 
\begin{equation}
\Q_{\theta}^r( A_{n,r}^{\mathbf{x}}) =  \mu_{\theta}(\mathbf{x}) \Q_{\theta}^r\left( \sum_{j=1}^{r-k} \mathcal{G}_j = n- |\mathbf{x}| -r \right)  .
\end{equation} A similar decomposition holds with respect to $\mathbf{y}$. Recall Lemma \ref{lem:PolymerInterpretation} to bound the first and second moment of $|\mathcal{G}_j|$, and that $|\mathcal{G}_j|$ has full support on $\N \cup \{0\}$ by Remark~\ref{rem:AllGenerations}. Using~\eqref{eq:ThetaConstructed}, a local central limit theorem for the random variables $(|\mathcal{G}_j|)_{j \in [r-k]}$ ensures that for any $\varepsilon>0$, there exists some $N=N(\varepsilon,\mathbf{x},\mathbf{y})$ such that for all $n\geq N$, 
\begin{equation}\label{eq:LocalCLT}
 1- \varepsilon \leq  \inf_{r \in I_n} \frac{ \Q_{\theta}^r\big( \sum_{j=1}^{r-k} \mathcal{G}_j = n- |\mathbf{x}| \big)}{\Q_{\theta}^r\big( \sum_{j=1}^{r-k} \mathcal{G}_j = n- |\mathbf{y}| \big)} \leq    \sup_{r \in I_n} \frac{ \Q_{\theta}^r\big( \sum_{j=1}^{r-k} \mathcal{G}_j = n- |\mathbf{x}| \big)}{\Q_{\theta}^r\big( \sum_{j=1}^{r-k} \mathcal{G}_j = n- |\mathbf{y}| \big)} \leq 1 + \varepsilon \, .
\end{equation} Moreover,  by \eqref{eq:ThetaWellEstimated}, the feature distribution $\mu_{\theta}$ with $\theta=\theta(n,r)$ satisfies for fixed $\mathbf{x}$ and $\mathbf{y}$
\begin{equation}\label{eq:FeatureConvergenceTheta}
\lim_{n\rightarrow \infty} \inf_{r \in I_n}\frac{\mu_{\theta}(\mathbf{x})}{\mu_{\theta}(\mathbf{y})} = \lim_{n\rightarrow \infty} \sup_{r \in I_n}\frac{\mu_{\theta}(\mathbf{x})}{\mu_{\theta}(\mathbf{y})} =  \frac{\mu_{\theta_{\ast}}(\mathbf{x})}{\mu_{\theta_{\ast}}(\mathbf{y})} \, .
\end{equation} Combining \eqref{eq:LocalCLT} and \eqref{eq:FeatureConvergenceTheta}, this gives \eqref{eq:FractionConvergence}, and thus finishes the proof.
\end{proof}


\section{Counting the number of colliders in a uniformly sampled \DAG}\label{sec:NumberColliders}

We now show that the number of non-collider edges $|C(G)|$, i.e.\ the number of edges which  are not part of any $v$-structure, in a uniformly chosen \DAGc  as well as the ratio between essential \DAGs and \MECs with respect to the total number of \DAGs on $n$ sites, converges as $n$ goes to infinity. Recall the tower decomposition from Section \ref{sec:HeightDecomposition},  and set
\begin{equation}
L(G) := \sup\left\{ i \geq 0 \colon \exists v \in H_{\hei(G)-i} \text{ such that } (u,v) \in C(G) \text{ for some } u\in V(G)  \right\}
\end{equation} to be the last tower generation containing a site of a non-collider edge. Here, we use the convention $L(G)=-1$ if each edge in $G$ is part of a collider, i.e.\ $G$ is essential.

\begin{lemma}\label{lem:ExpectedNonCollider} 
There exists some non-negative, almost surely finite random variable $X$ such that for all $s\geq 0$
\begin{equation}\label{eq:CountingLastColliderGeneration}
\limsup_{n \to \infty} \P_{n}\left( L(G) \geq s \right) \leq \P(X \geq s)  .
\end{equation}
Moreover, there exists a non-negative, almost surely finite $Y$ such that for all $s\geq 0$
\begin{equation}\label{eq:CountingNonCollidersUpper}
\limsup_{n \to \infty} \P_{n}( |C(G)| \geq s) \leq \P(Y \geq s)  .
\end{equation}  
\end{lemma}
\begin{proof} Recall that we assign for each \DAG $G$ a unique tower vector $h=(h_{1},h_{2},\dots,h_{\hei(G)})$. For all $\varepsilon>0$, we define the event
\begin{equation}
\mathcal{B}_{\varepsilon} := \bigcup_{i \in \N} \big\{ h_{(\hei(G)-i)} \geq \log(i) + 2 \log(\varepsilon^{-1}) +5 \big\}\, ,
\end{equation} with the convention that $h_{j}=0$ for all $j<1$. By Proposition \ref{pro:LevelEstimate}, we get that
\begin{equation}
\limsup_{n \rightarrow \infty}\P_n( \mathcal{B}_{\varepsilon}) \leq  \varepsilon \exp(-5) \sum_{i=1}^{\infty} \exp(-\log^{2}(i)) < \varepsilon \, .
\end{equation} Thus, we get the lower bound 
\begin{equation}\label{eq:AllGenerationsGood}\liminf_{n \rightarrow \infty}\P_n( h_{(\hei(G)-i)} \leq \log(i) + 2 \log(\varepsilon^{-1}) +5 \text{ for all } i\in \N ) \geq 1- \varepsilon \, 
\end{equation} Let $\mathcal{B}^{\textup{c}}_{\varepsilon}$ denote the complement of the event $\mathcal{B}_{\varepsilon}$. Recall that we have for each pair of sites $v_i\in H_i$ and $v_j \in H_j$ with $i \geq j + 2$ that the edge $e=(v_i,v_j)$ is present independently with probability $\frac{1}{2}$. Thus, for each $v\in H_{\hei(G)-k}$ and $w\in H_{\hei(G)-j}$ for some $k\geq 3$ and $j<k$, 
\begin{equation}
\limsup_{n \rightarrow \infty}\P_n( (w,v) \in  C(G) \, | \, \mathcal{B}^{\textup{c}}_{\varepsilon} ) \leq \Big( \frac{3}{4} \Big)^{k-3}  .
\end{equation}
We obtain that for all $j\in \N$ and $\varepsilon>0$
\begin{align*}
\limsup_{n \rightarrow \infty}\P_n( L(G) \geq j ) &\leq   \limsup_{n \rightarrow \infty} \P_n( L(G) \geq j \,  | \, \mathcal{B}^{\textup{c}}_{\varepsilon} ) + \varepsilon \leq \varepsilon + \Big(2 \log\Big(\frac{1}{\varepsilon}\Big) +5\Big)^{2} \sum_{k \geq j}k\Big( \frac{3}{4} \Big)^{k-3}  
\end{align*} using a union bound. Choosing now $j=j(\varepsilon)$ sufficiently large, this gives the first bound \eqref{eq:CountingLastColliderGeneration}. For the second bound \eqref{eq:CountingNonCollidersUpper}, note that 
 \begin{equation}
\left\{ |C(G)| \leq \prod^{j}_{i=1} (\log(i)+2 \log(\varepsilon^{-1}) +5)  \right\} \supseteq  \{ L(G) \leq j \} \cap \mathcal{B}^{\textup{c}}_{\varepsilon} \, 
 \end{equation} which, when combined with \eqref{eq:CountingLastColliderGeneration} and \eqref{eq:AllGenerationsGood},  allows us to conclude.
\end{proof}

Next, we show that with positive $\P_n$-probability, uniformly in $n$, we see in fact only collider edges in the sampled \DAGp
\begin{lemma}\label{lem:SeeNoCollider}
It holds that
\begin{equation}
\liminf_{n \to \infty} \P_{n}( |C(G)| = 0) > 0 \, .
\end{equation}
\end{lemma}
\begin{proof}
Recall that $\mathcal{D}_n$ denotes the set of all labeled \DAGs on  $n$ sites. By Lemma \ref{lem:ExpectedNonCollider}, and the fact that a uniformly drawn \DAG from $\mathcal{D}_n$ has with probability tending to $1$ as $n$ goes to infinity no isolated sites, there exists some $J>0$ such that
\begin{equation*}
\mathcal{D}^{J}_n :=  \big\{ G \in \mathcal{D}_n \colon L(G) < J \text{ and } |h_{\hei(G)-k}| \leq J  \  \forall  k\in [J] \text{ and } \textup{deg}(v)\geq 1 \  \forall v \in V(G) \big\} 
\end{equation*} satisfies $|\mathcal{D}^{J}_n| \geq \frac{1}{2}|\mathcal{D}_n|$ for all $n$ sufficiently large. Consider now the mapping $f \colon \mathcal{D}^{J}_n \mapsto \mathcal{D}_n$ where each \DAG $G$ in the set $\mathcal{D}^{J}_n$ gets mapped to a \DAG $G^{\prime}$, which we obtain from $G$ by removing all edges $e=(v,w)$ with $v\in h_{\hei(G)-i}$ and $w\in h_{\hei(G)-j}$ for some $i,j \leq J$. We claim that for each graph $G \in \mathcal{D}^{J}_n$, the corresponding graph $G^{\prime}=f(G)$ has only collider edges. To see this, note that we remove all non-collider edges in $G$ in this way, while in this construction no new non-collider edges are created. Note that $|f^{-1}(G^{\prime})| \leq 2^{J^4}$  for all $G^{\prime}\in \mathcal{D}_n$ as we have at most $J^2$ many sites, and hence at most $J^4$ many different edges in the last $J$ layers of any graph $G \in \mathcal{D}^{J}_n$, while no other vertices yield isolated sites in~$G^{\prime}$. 
Thus, the $\P_n$-probability of $|C(G)|=0$ is at least $\frac{1}{2} 2^{-J^4}>0$ for all $n$ sufficiently large.
\end{proof}

We have now all tools to show Theorem \ref{thm:Essential} and Theorem \ref{thm:AverageEssentialMEC}.

\begin{proof}[Proof of Theorem \ref{thm:Essential} and Theorem \ref{thm:AverageEssentialMEC}]
We start by showing that the limits in \eqref{eq:CountingNonColliders} and \eqref{eq:SeeingNonColliders}, as well as the limits in \eqref{eq:AverageMEC} and \eqref{eq:AverageEssential} exist. Fix some $\varepsilon>0$, and recall that by Lemma \ref{lem:ExpectedNonCollider}, there exists some $J=J(\varepsilon)$ such that $\P_{n}(L(G)\geq J)< \varepsilon$ for all $n$ sufficiently large. Further, recall from Lemma \ref{lem:GadgetSize} that the law of the last $J$ layers under $\P_n$ converges when $n$ goes to infinity. As a consequence, we see that the law of the edges between all sites in the last $J$ layers on the tower converges to some law $\nu_J$.  On the event $\{ L(G) < J\}$, the Markov equivalence class of a \DAG $G$ is fully determined by its  subgraph of the last $J$ layers, and we obtain the existence of the constants $c_{\textup{MEC}}$ and $c_{\textup{Ess}}$ in \eqref{eq:AverageMEC} and \eqref{eq:AverageEssential}. It remains to argue that $c_{\textup{MEC}}, c_{\textup{Ess}} \in (0,1)$. We will only argue that $c_{\textup{Ess}}>0$ as the remaining cases are similar. Let $C_{\ast}$ be the number of non-collider edges according to $\nu_J$.
Notice that on the event $\{ L(G) < J\} $, we have that $C(G)$ depends only on the edges between sites in the last $J$ layers of the tower. Thus,  using Lemma \ref{lem:ExpectedNonCollider}, for each $s\geq 0$
\begin{equation}\label{eq:ConvergenceRatios}
\limsup_{n \rightarrow \infty} |\P_n( |C(G)|\geq s) - \nu_J( C_{\ast} \geq s) | \leq \varepsilon \, .
\end{equation} Moreover, using Lemma \ref{lem:SeeNoCollider}, we obtain that for any $\varepsilon>0$, we can choose $J=J(\varepsilon)$ sufficiently large such that 
\begin{equation}
c_{\textup{Ess}} \geq \nu_{J}( C_{\ast} = 0) - \varepsilon  > 0
\end{equation} holds, which finishes the proof.
\end{proof}

\section{Representations of Markov equivalence classes}\label{sec:RepresentationOfMECs}

In this section, we collect facts and estimates on directed acyclic graphs and the size of Markov equivalence classes, which we will use in the proof of Theorem \ref{thm:Moments}. We start with a poset decomposition, which can be found for example in Section~3 of \cite{BT:NoteMECS}.

\subsection{A poset decomposition for \DAGs} \label{sec:PosetDecomposition}

For a given \DAG $G$, we define a partial order on its vertex set $V$ as follows. We let $v \succeq w$ for $v,w\in V$ if and only if there exists a directed path in $G$ from $w$ to $v$. We call the resulting poset $(V,\succeq)$ the \textbf{reachability poset} of $G$. We say that $v$ \textbf{covers} $w$ if $v \succeq w$, and there is no $u\in V$ such that $v \succeq u \succeq w$. For each $v\in V$, let $\text{dom}(v)$ be the number of sites below $v$ according to $\succeq$, and let $\text{cov}(v)$ be the number of sites covered by $v$. For a given poset $(V,\succeq)$, we set 
\begin{equation}\label{def:PosetWeights}
w_{\succeq}:= \prod_{v\in V} 2^{\text{dom}(v)-\text{cov}(v)}  \, ,
\end{equation} and we partition the \DAGs according to their reachability posets; see Proposition 16 in~\cite{BT:NoteMECS}. For the sake of completeness, we include the following lemma on the number of  \DAGs for a given reachability poset. 

\begin{lemma}\label{lem:PosetDecomposition}
For a given labeled poset $(V,\succeq)$, there are exactly $w_{\succeq}$ many \DAGs which have $(V,\succeq)$ as their reachability poset, and the set of all reachability posets partitions the set of labeled directed acyclic graphs.
\end{lemma}
A similar relation is given in Proposition 17 of  \cite{BT:NoteMECS} to count the number of essential \DAGsp In particular, note that we call a vertex $v$ a source, respectively a sink, if it is a minimal, respectively a maximal, element in the reachability poset.

\subsection{Partially directed acyclic graphs} \label{sec:PDAGs}

Note that every \MEC can be represented as a partially directed acyclic graph (\textsf{PDAG}), i.e.\ the edges may be either directed or undirected such that the graph contains no directed cycles. For a given \DAG $G$, we obtain its corresponding \PDAG by keeping the order of an edge if and only if it has the same orientation in each element of $M(G)$, and let it be undirected, otherwise. While we always obtain a \PDAG in this construction, Andersson et al. provide necessary and sufficient conditions for a given partially directed acyclic graph to yield a Markov equivalence class \cite{AMP:MECsForDAG}. These \PDAGs corresponding to a specific Markov equivalence class are also called \textbf{essential graphs}. The following lemma is a simple observation between non-collider edges and Markov equivalence classes using essential graphs.
\begin{lemma} \label{lem:PDACnoncollider}
Let $e$ be an undirected edge in an essential graph. Then for all \DAGs in the corresponding Markov equivalence class, we have that $e$ is a non-collider edge.
\end{lemma}
\begin{proof}
This follows immediately from the construction of the essential graph since every edge which is part of a $v$-structure has the same orientation in every \DAG of its Markov equivalence class.
\end{proof}
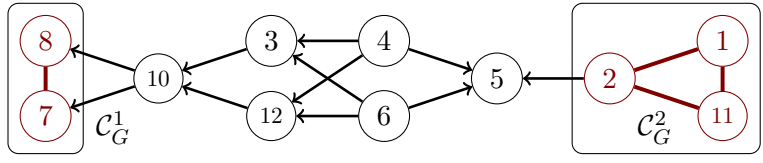
\begin{figure}
\begin{center}
\begin{tikzpicture}[scale=0.5]

	\node[shape=circle,scale=1,draw,my-red] (A) at (-6,-1){$7$} ;
 	\node[shape=circle,scale=1,draw,my-red] (B) at (-6,1){$8$} ;
	\node[shape=circle,scale=0.8,draw] (C) at (-3,0){$10$} ;
 	\node[shape=circle,scale=1,draw] (D) at (0,1){$3$} ;	
 	\node[shape=circle,scale=0.8,draw] (E) at (0,-1){$12$} ;
	\node[shape=circle,scale=1,draw] (F) at (3,1){$4$} ;
 	\node[shape=circle,scale=1,draw] (G) at (3,-1){$6$} ; 
 	\node[shape=circle,scale=1,draw] (I) at (6,0){$5$} ;
	\node[shape=circle,scale=1,draw,my-red] (J) at (9,0){$2$} ;
 	\node[shape=circle,scale=1,draw,my-red] (K) at (12,1){$1$} ;	 	
 	\node[shape=circle,scale=0.8,draw,my-red] (L) at (12,-1){$11$} ;
 	 	  	
 	  	\draw[line width=1,->] (C) to (A);		
 	  	\draw[line width=1,->] (C) to (B);	
 	  	\draw[line width=1.5,my-red] (A) to (B);	
 	  	
 	  	\draw[line width=1,->] (D) to (C);		
 	  	\draw[line width=1,->] (E) to (C);
 	  	\draw[line width=1,->] (G) to (D);

 	  	\draw[line width=1,->] (F) to (D);		
 	  	\draw[line width=1,->] (F) to (E);	
 	  	\draw[line width=1,->] (G) to (E);		
 
 	  	\draw[line width=1,->] (G) to (I);		
 	  	\draw[line width=1,->] (F) to (I);	
 	  	\draw[line width=1,->] (J) to (I);	  		

 	  	\draw[line width=1.5,my-red] (K) to (J);
 	  	\draw[line width=1.5,my-red] (L) to (J);	 
		\draw[line width=1.5,my-red] (K) to (L);	  	  
 	  	
\draw[rounded corners] (-7,-2) rectangle (-5,2) {}; 	 

\node[shape=circle,scale=1] (AX) at (-4.2,-1.3){$\mathcal{C}^1_G$} ;

\draw[rounded corners] (8,-2) rectangle (13,2) {}; 	  

\node[shape=circle,scale=1] (AX) at (10.2,-1.3){$\mathcal{C}^2_G$} ;
 \end{tikzpicture}
\end{center}
\caption{\label{fig:ChainGraph} Decomposition of an essential graph for a \DAG $G$ into its chain graph components $\mathcal{C}^1_G$ and  $\mathcal{C}^2_G$,  drawn in red.} 
\end{figure}
It will be convenient to use a decomposition for essential graphs  when estimating the expected size of a Markov equivalence class. For a given \DAG $G$, we define its \textbf{chain graph} $\mathcal{C}_G=(V(\mathcal{C}),E(\mathcal{C}))$ by removing all directed edges in the essential graph of its Markov equivalence class $M(G)$, and keeping only vertices which are adjacent to at least one undirected edge. Note that $V(\mathcal{C}) \subseteq V(G)$, and we denote the connected components (of size at least $2$) of the chain graph as $\mathcal{C}^1_G,\mathcal{C}^2_G,\dots,\mathcal{C}^{\ell}_G$ for some $\ell\in \N$. A visualization of this construction can be found in Figure~\ref{fig:ChainGraph}. We recall the following observation from \cite{GP:SizeDistribution}.
\begin{lemma}\label{lem:MECcountDecomposition}
For a given \DAG $G$, we have that
\begin{equation}
|M(G)| = \prod_{i=1}^{\ell} \big|M(\mathcal{C}^{i}_G)\big|
\end{equation} where $|M(\mathcal{C}^{i}_G)|$ is the number of ways to orient the edges within $\mathcal{C}^{i}_G$ such that they agree with the edge orientation of some \DAG in $M(G)$.
\end{lemma}
As a simple consequence of Lemma \ref{lem:PDACnoncollider} and Lemma \ref{lem:MECcountDecomposition}, we obtain the following statement.

\begin{lemma}\label{lem:MECviaColliders} For a given \DAG $G$, let $N(G)$ be the set of vertices which are part of some non-collider edge. Then it holds that \begin{equation}
|M(G)| \leq |N(G)|! \, .
\end{equation}
\end{lemma}
\begin{proof}
Consider the representation of a \MEC of a given \DAG $G$ as a partially directed acyclic graph, where the directed edges have the same orientation as in $G$. Let $(\mathcal{C}^{i}_G)_{i\in [\ell]}$ for some $\ell\in \N$ denote the connected (undirected) components of its chain graph, and $|\mathcal{C}_G^{i}|$ the number of sites in $\mathcal{C}_G^{i}$. Note that by the poset decomposition in Lemma \ref{lem:PosetDecomposition}, for each $\mathcal{C}_G^{i}$, there are at most $|\mathcal{C}_G^{i}|!$ ways to orient the edges. Together with Lemma~\ref{lem:MECcountDecomposition}, this yields
\begin{equation}
|M(G)| \leq \prod_{i=1}^{\ell} |\mathcal{C}_G^{i}|! \ .
\end{equation} Since every site in a component $\mathcal{C}_G^{i}$ must be an element of $N(G)$ by Lemma \ref{lem:PDACnoncollider}, and the connected components $(\mathcal{C}^{i}_G)_{i\in [\ell]}$ are vertex disjoint by construction, we conclude.
\end{proof}

Lemma \ref{lem:MECviaColliders} gives good bounds when the number of non-collider edges is large compared to the number of sites. However, we will require a refined estimate on the number of Markov equivalence classes when the vertices of the chain graph can be split into two parts: a core, where only a small number of sites are connected to other sites outside of the core using non-collider edges, and a remainder where no two sites are connected by non-collider edges. This is formalized in the following proposition which we require as a key took for the proof of Theorem~\ref{thm:Moments}. In order to state it, we fix in the following unique labels $\ell(v) \in [n]$ for all sites $v \in V(G)$, which are consistent with the tower structure, i.e.\ such that $\ell(v) \geq \ell(w)$ holds if and only if $v\in H_i$ and $w \in H_j$ with $i \leq j$ in the tower $H(G)$. 

\begin{proposition}\label{pro:RefinedMECcounting} For a given directed acyclic graph $G$, let $\mathcal{C}_{G}=(V(\mathcal{C}),E(\mathcal{C}))$ denote its chain graph. Recall that  $V(\mathcal{C}) \subseteq V(G)$, and suppose that $V(\mathcal{C})$ can be partitioned as $(V_1,V_2)$ such that the following four properties hold for some natural numbers $a,b>0$. 
\begin{itemize}
\item[1.] Each layer of the tower $H$ of $G$ with a site in $V(\mathcal{C})$ consists of at most $a$ many sites.
\item[2.] There are at most $b$ many sites in $V_1$ which are adjacent to a site in $V_2$ using $E(\mathcal{C})$.
\item[3.] We have $\{v,w \} \notin E(\mathcal{C})$ for all $v,w \in V_2$.
\item[4.] For each pair of sites $v,w \in V_2$ with $v\in H_{i}$ and $w \in H_j$ for some $i \geq j+3$ such that $\sum_{k=j+1}^{i-1} h_j \geq 3a$, there exists a directed path in $G$ from $v$ to $w$.
\end{itemize}
Then we have that the number $|M(G)|$ of \MECs with respect to $G$ satisfies
\begin{equation}
|M(G)| \leq |V_1| ! (|V_2|(10a)!)^b \, .
\end{equation}
\end{proposition}

\begin{proof} Let $\tilde{V}$ be the set of sites in $V_1$ that are connected to some site in $V_2$ via an undirected edge, and note that $|\tilde{V}| \leq b$ by our second assumption. By Lemma \ref{lem:MECviaColliders}, there are at most $|V_1|!$ many ways to orientate the edges within $V_1$ in order to obtain an element of the \MECp For each such ordering, it remains to show that there are at most $(|V_2|(10a)!)^b$ valid ways to assign the directions of the remaining edges in $E(\mathcal{C})$ between $V_1$ and $V_2$. \\

To this end, our key observation is that for each $v\in \tilde{V}$, there exists a transition point $w\in V_2$ in the direction of the outgoing edges from $v$ to $V_2$. More precisely, fix an orientation of the edges in $V_1$ leading to a \DAGc and a vertex $v\in \tilde{V}$. For a fixed orientation of the remaining non-collider edges, let $O(v)$ be the smallest label of a site $w_1$ such that $(v,w_1)$ is present for some $\{v,w_1\} \in E(\mathcal{C})$, and let $I(v)$ be the largest label of a site $w_2$ such that the edge $(w_2,v)$ is present for some $\{w_2,v\} \in E(\mathcal{C})$, with conventions $I(v)=0$ and $O(v)=n+1$ if there are no such edges. We claim that $I(v)-O(v) \leq 10a$. To see this, note that by the first and forth assumption, and the definition of the labels $\ell$, for each pair of edges $(x,v)$ and $(v,y)$ with $x$ for $\ell(x)\leq \ell(v)-5a$ and with $y$ for $\ell(y)\geq \ell(v)+ 5a$, we would create a directed cycle if $I(v)-O(v)>10a$. \\

Now for every $v\in \tilde{V}$, note that there are at most $|V_2|$ many ways to choose $I(v)$. For each such choice, there are at most $10a$ many sites in $V_2$ to which the ordering of edges between $\tilde{V}$ and $V_2$ is not determined from knowing $I(v)$. Fixing now one of remaining at most $(10a)!$ many ways to orient the edges between $v$ and these at most $10a$ many sites, 
our third assumption guarantees that this uniquely characterizes the respective \DAG in $M(G)$. This yields the desired claim, and thus finishes the proof.
\end{proof}

\section{Counting the number of \MECs of a uniformly sampled \DAG}\label{sec:NumberMECs}

We now investigate the set of vertices $N(G)$ of a \DAG $G$ sampled according to $\P_n$, which are adjacent to some non-collider edge in $C(G)$. Recall from Lemma~\ref{lem:MECviaColliders} that the size of a \MEC for $n$ non-collider sites is at most $n!$. Note that this bound is tight when the underlying graph is a clique, but can in general be quite rough, for example for tree-like graphs. Therefore, we use a decomposition of the sites $N(G)$ in order to obtain an improved bound on the size of the \MEC of $G$,  applying Proposition \ref{pro:RefinedMECcounting}.  
We give now the setup for Theorem \ref{thm:Moments}. Our main idea is to partition the sites $N(G)$ according to their position within the tower and their connection properties to obtain the following proposition. 

\begin{proposition}\label{pro:NonCollider} Recall that $M(G)$ denotes the \MEC of $G$. Then there exists some universal constants $c,N>0$ such that for all $n \geq N$, and all $m\in [n^{19/20}]$.
\begin{equation}
\P_{n}( |M(G)| \geq m!) \leq \exp\big(-cm^{\frac{11}{10}}\big) \, .
\end{equation} 
\end{proposition}

\begin{proof}[Proof of Theorem \ref{thm:Moments} using Proposition \ref{pro:NonCollider}] Note that $|M(G)| \leq n!$ holds $\P_n$-almost surely. Since for $t=m!$ with $m \in \N$ large enough, we have
\begin{equation}
\frac{\log(t)}{2 \log\log(t)} \leq m \leq \log(t) \, ,
\end{equation} the claimed bound on the tail of $|M(G)|$ follows from Proposition \ref{pro:NonCollider}.
\end{proof}

In the remainder, our goal is to show Proposition \ref{pro:NonCollider}. We achieve this in four steps. First, we consider the tower decomposition and bound the probability that we see a non-collider edge, which reaches a site with a large label, recalling the labeling $\ell$ defined before Proposition \ref{pro:RefinedMECcounting}. We then identify bad vertices, which have undesirable connection properties, and show that with sufficiently large  probability, their number is not too large. In a next step, we bound the total number of bad sites which are connected by non-collider edges to some good site. In a last step, we combine the above estimates to conclude.

\subsection{On the total number of non-collider edges deep in the tower}

We estimate the number of non-collider edges with both endpoints deep in the tower. Recall the tower decomposition $H(G)$ with a tower vector $h(G)$. Further, recall the labeling $\ell$ before Proposition \ref{pro:RefinedMECcounting}, which is consistent with the tower decomposition, i.e.\ $\ell(v) \geq \ell(w)$ holds if and only if $v\in H_i$ and $w \in H_j$ for some $i \leq j$. We give a bound on the number of such edges where the lower endpoint has a label between $k/3$ and $k^{7/6}$, and show that with sufficiently large probability, there are no such non-collider edges to a site with label at least $k^{7/6}$. Before doing so, we start with a bound on the number of sites in each layer in the tower decomposition. We define the event
\begin{align}\label{def:B1}
\mathcal{B}_1 := \left\{  |h_{T(h)+1-i}| \leq k^{2/3} \text{ for all } i\in [k^2] \right\} \cap  \left\{ |h_{T(h)+1-i}| \leq i^{2/3} \text{ for all } i > k^{2} \right\}
\end{align} Here, we use the convention that $h(i)=0$ if $i>T(h)$, where $T(h)$ is the length of $h$.

\begin{lemma}\label{lem:NumberOfSitesPerGeneration}
There exist constants $c,N>0$ such that for all $n \geq N$ and $k\in [n]$, we have that $\mathcal{B}_1=\mathcal{B}_1(k)$ holds with $\P_n$-probability at least $1-\exp(-ck^{5/4})$ for all $k>0$ sufficiently large.
\end{lemma}
\begin{proof} This follows from a simple computation using Lemma \ref{pro:LevelEstimate}, and summing over all possible values of the entries in  $h$.
\end{proof}
Let us stress that in the following, we first sample the tower vector $h$, and then the respective \DAG accordingly, recalling Remark \ref{rem:IndependentSamplingTower}. We conclude this subsection showing that with sufficiently high probability, we see  only collider edges deep in the tower.

\begin{lemma}\label{lem:ColliderHighInTower}
There exist constants $c,N>0$ such that for all $n \geq N$ and $k\in [n]$, we have 
\begin{equation*}
 \P_{n}\left( \exists e=(v,w) \text{ non-collider  with  } i,j \geq k^{\frac{7}{6}} \text{ with } \ell(v)=i \text{ and } \ell(w)=j\right)  \leq \exp(-ck^{11/10}) \, .
\end{equation*} 
\end{lemma}
\begin{proof}
Condition on the event $\mathcal{B}_1$ and consider the events $\mathcal{A}_i$ that there is a non-collider edge between sites with labels between $k^{7i/6}$ and $k^{7(i+1)/6}$, i.e.\ we set
\begin{equation}\label{eq:PartitionTheSites}
\mathcal{A}_i := \left\{ \exists (v,w)\in C(G) \colon \ell(w) \in \Big[ k^{7i/6},k^{7(i+1)/6} \Big]\right\}
\end{equation}
for all $i\in \N$. We claim that $\mathcal{A}_i$ has $\P_n$-probability at most $\exp(-c_0k^{7i/6})$ for some universal constant $c_0>0$. Since we condition on $\mathcal{B}_1$, and there are at most $k^{7(i+1)/3}$ edges $e=(v,w)$ which we have to consider for $\mathcal{A}_i$. Note that with probability at least $1- \exp(-c_1k^{7i/6})$ for some constant $c_1>0$, there are at least $k^{7i/6}/9$ sites in $U_v=\{u \colon \ell(u) \leq k^{7i/6}/4 \}$ which are connected to $w$. All edges from a site  in $U_v$ to the site $v$ must be present in order to have $e$ as a non-collider. Recalling Remark \ref{rem:IndependentSamplingTower}, this shows that for all $n$ and $k$ large enough
\begin{equation}
\P_{n}( (v,w) \in C(G)) \leq \exp(-c_2 k^{7i/6})
\end{equation} for some constant $c_2>0$. We obtain \eqref{eq:PartitionTheSites} by a union bound over all admissible edges. Using another union bound on the events $\mathcal{A}_i$ over $i\in \N$, this finishes the proof.
\end{proof}

\subsection{Construction of good and bad vertices}

In order to give the definition of bad vertices, we use the following additional notation. For $B \subseteq [n]$, we write $\deg(v, B)$ for the number of edges between $v$ and a site $w$ with $\ell(w) \in B$. \\ 

For $k \in \N$, we say that $v\in V$ with $\ell(v) \in [k/3,k^{7/6}]$ is $\boldsymbol{k}$-\textbf{good} if the following holds:
\begin{itemize}
 \item[(1)] For all $i \in [4 k^{4/5},k/3 ] \cap \Z$, and $j \in [4 k^{4/5},k/3 ] \cap [n-\ell(v)]$, we have that
 \begin{equation*}
 \deg\left(v, [\ell(v)+1,\dots,\ell(v)+j]\right) \geq \frac{j}{3} \quad \text{ and } \quad \deg\left(v, [\ell(v)-i\dots,\ell(v)-1]\right) \geq \frac{i}{3}
 \end{equation*}
 \item[(2)] 	The relative ordering of the site $v$ is fully determined for all sufficiently distant sites, i.e.\ for all $w$ with $\ell(w)\geq \ell(v)+5k^{\frac{4}{5}}$ such that $(v,w) \in C(G)$, we can find some  $u\in V$ with \begin{equation*}
 (v,u),(u,w) \in E \setminus C(G) \, .
 \end{equation*} Similarly when $(w,v) \in C(G)$ and $\ell(w)\leq \ell(v)-5k^{\frac{4}{5}}$, there is some  $u\in V$ with \begin{equation*}
 (w,u),(u,v) \in E \setminus C(G) \, ,
 \end{equation*} i.e.\ a path of length two of collider edges from $w$ to $v$.
 \item[(3)] The number of connections to the bottom $k/20$ edges is typical, i.e.\ 
 \begin{equation*}
 \deg\left(v, \left[0,\dots,\frac{k}{20}\right]\right)  \in \left[\frac{k}{80}, \dots, \frac{3k}{80} \right] \, .
 \end{equation*}
 \item[(4)] The number of connections to the sites with label between $k/5$ and $k/4$  is typical, i.e.\ 
 \begin{equation*}
 \deg\left(v, \left[\frac{k}{5},\dots,\frac{k}{4}\right]\right) \in \left[\frac{k}{80}, \dots, \frac{3k}{80} \right] \, .
 \end{equation*} 
\end{itemize}
All sites $v$ with $\ell(v) \in [k/3,k^{7/6}]$, which are not $k$-good are called $\boldsymbol{k}$-\textbf{bad}. In the sequel, we bound in the total number of $k$-bad sites, and the number of connections using non-collider edges to some $k$-good site; see also Figure \ref{fig:GoodAndBad} for a visualization.
\begin{figure}
\begin{tikzpicture}[scale=0.6]

\node[shape=circle,fill,draw,scale=0.6] (A1) at (8,1){} ;
\node[shape=circle,fill,draw,scale=0.6] (A2) at (12,0){} ;
\node[shape=circle,fill,draw,scale=0.6] (A3) at (14.5,0.5){} ;
\node[shape=circle,fill,draw,scale=0.6] (A4) at (6.8,0.7){} ;
\node[shape=circle,fill,draw,scale=0.6] (A5) at (10.1,-1.4){} ;
\node[shape=circle,fill,draw,scale=0.6] (A6) at (11,1.5){} ;
\node[shape=circle,fill,draw,scale=0.6] (A7) at (13.2,0.2){} ;
\node[shape=circle,fill,draw,scale=0.6] (A8) at (7.4,-1.2){} ;

\node[shape=circle,fill,draw,scale=0.6,my-red!50] (A9) at (1,-1.3){} ;
\node[shape=circle,fill,draw,scale=0.6] (A10) at (-3,0.8){} ;
\node[shape=circle,fill,draw,scale=0.6,my-red!50] (A11) at (0.2,-0.3){} ;
\node[shape=circle,fill,draw,scale=0.6] (A12) at (-1.8,-1.1){} ;
\node[shape=circle,fill,draw,scale=0.6] (A13) at (-5,-0.7){} ;
\node[shape=circle,fill,draw,scale=0.6] (A14) at (-2,-0.5){} ;
\node[shape=circle,fill,draw,scale=0.6,my-red!50] (A15) at (1.2,1.6){} ;
\node[shape=circle,fill,draw,scale=0.6] (A16) at (-6.3,1.1){} ;

	\draw[line width=1,bend left=20] (A15) to (A1);	
	\draw[line width=1,bend left=20] (A15) to (A4);	
	\draw[line width=1,bend left=23] (A15) to (A2);	 
	\draw[line width=1,bend left=20] (A15) to (A6);	 

	\draw[line width=1,bend left=20] (A9) to (A8);	
	\draw[line width=1,bend left=23] (A11) to (A5);		 
	\draw[line width=1,bend left=20] (A11) to (A8);

\draw[rounded corners] (-7,-2) rectangle (2,2) {}; 	 

\node[scale=1] (AX) at (-2.5,-2.7){$k$-bad sites $V_1$} ;

\draw[rounded corners] (6,-2) rectangle (15,2) {}; 	  

\node[scale=1] (AX) at (10.5,-2.7){$k$-good sites $V_2$};

 \end{tikzpicture}
\caption{\label{fig:GoodAndBad} Visualization of the decomposition into $k$-good and $k$-bad sites. We only draw the non-collider edges between $k$-good and $k$-bad sites. In order to apply Proposition \ref{pro:RefinedMECcounting}, a main task is to bound the number of $k$-bad sites, drawn in red, which have a least one edge to some $k$-good site.} 
\end{figure}
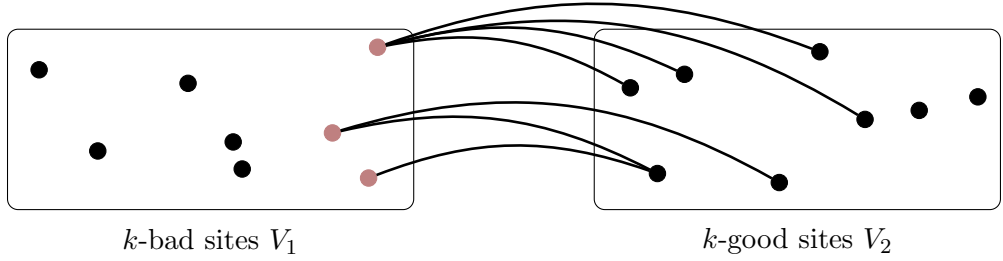
\subsection{On the total number of bad vertices}

In this subsection, we estimate the total number of bad sites. The following lemma is the key result of this subsection. 

\begin{lemma}\label{lem:BadSites}
There exist some universal constants $c,N>0$ such that for all $n \geq N$, and $k\in [n^{19/20}]$, with probability at least $1-\exp(-c k^{\frac{11}{10}})$ at most $k/100$ sites are $k$-bad.
\end{lemma}

\begin{proof} Observe that on $\mathcal{B}_1$ from \eqref{def:B1}, we have for every site $v$ with $\ell(v) \in [k/30,n-k^{4/5}]$ that each edge to a site $w$ with $\ell(w)< \ell(v)- k^{2/3}$ or $\ell(w)> \ell(v)+ k^{2/3}$ is present independently with probability $\frac{1}{2}$. Thus we see that for $n$ and $k$ sufficiently large the events 
\begin{align}\label{eq:LargeUpDegree}
 \mathcal{B}_v^{\text{up}} &:= \Big\{ \deg\left(v, [\ell(v)+1,\dots,\ell(v)+k^{4/5}]\right) \leq \frac{k^{4/5}}{3} \,   \Big| \,\mathcal{B}_1 \Big\} \leq \exp(-c k^{3/4}) \\
 \label{eq:LargeDownDegree}
\mathcal{B}_v^{\text{down}} &:= \Big\{ \deg\left(v, [\ell(v)-k^{4/5}\dots,\ell(v)-1] \right)\leq \frac{k^{4/5}}{3}   \,   \Big| \, \mathcal{B}_1 \Big\} \leq   \exp(-ck^{3/4})
\end{align} hold with probability at most $\exp(-ck^{3/4})$ for some constant $c>0$. 
Since the events $\mathcal{B}_v^{\text{up}}$ and $\mathcal{B}_v^{\text{down}}$ in \eqref{eq:LargeUpDegree} and \eqref{eq:LargeDownDegree} are independent for different sites $v$, respectively, this shows that the events
\begin{equation}
\mathcal{B}^{\text{up}}_2 := \left\{ \left| v\in V \colon \ell(v)\in [k/3,\min(n-k^{4/5},k^{11/10})] \text{ and } \mathcal{B}_v^{\text{up}} \text{ holds} \,  \right| \leq \frac{k^{4/5}}{2} \right\}
\end{equation} as well as
\begin{equation}
\mathcal{B}^{\text{down}}_2 := \left\{ \left| v\in V \colon \ell(v)\in [k/3,\min(n-k^{4/5},k^{11/10})] \text{ and } \mathcal{B}_v^{\text{down}} \text{ holds} \,  \right| \leq \frac{k^{4/5}}{2} \right\}
\end{equation}
hold with probability at least $1 - \exp(-c_0 k^{\frac{11}{10}})$ for some universal constant $c_0$, using the exponential Markov inequality. In total, let $\mathcal{B}_2$ be the event that at most $2 k^{4/5}$ sites with $\ell(v)\in [k/3,k^{11/10}]$ do not satisfy property (1) in the definition of $k$-bad sites, and note that for all $n$ and $k$ sufficiently large, we get that for some constant $c_1>0$
\begin{equation}
\P_{n}(\mathcal{B}_2) \geq 1 - \frac{1}{5}\exp(-c_1 k^{\frac{11}{10}})  .
\end{equation} 
Next, in order to estimate the number of vertices which violate property (2), we condition on the event $\mathcal{B}_{1} \cup \mathcal{B}_{2}$ to hold. Let $E_2$ be a fixed set of $k^{1/2}$ many edges between sites $v,w$, which satisfy property (1). We claim that for all $n$ and $k$ sufficiently large
\begin{equation}\label{eq:CountE2}
\P_n( e \text{ does not satisfy (2) for all }e \in E_2 ) \leq \frac{1}{5 }\exp(-c_2 k^{\frac{11}{10}})
\end{equation} holds for some universal constant $c_2>0$. With a union bound over at most $\binom{k^2}{k^{1/2}}$ choices of $E_2$, this yields that there are with $\P_n$-probability at least $1-\exp(-c_3 k^{11/10})$ for some universal constant $c_3>0$ at most $2k^{1/2}$ sites which do not satisfy property (2). 
To see that \eqref{eq:CountE2} holds, we distinguish three cases. Suppose that  there exists some site $v$ and a set of sites $W_v$ with $|W_v|\geq k^{1/4}$ such that $(w,v) \in E_2$ for all $w \in W_v$. Then by property (1), there are at least $k^{4/5}$ sites $\tilde{W}_v$ with $\ell(u) \geq \ell(v)-4k^{4/5}$ for all $u\in \tilde{W}_v$ adjacent to $v$. Whenever the event in \eqref{eq:CountE2} holds, all edges of the form $(w,u)$ for $w\in W_v$ and $u\in \tilde{W}_v$ must be absent, giving us the desired bound. The same argument applies when there exists some site $v$ and a set of sites $W^{\prime}_v$ with $|W^{\prime}_v|\geq k^{1/4}$ such that $(v,w) \in E_2$ for all $w \in W^{\prime}_v$. In the remaining case, we have at least $k^{1/4}$ disjoint edges in $E_2$, and again use that the edge endpoints satisfy property (1) to conclude \eqref{eq:CountE2}. \\

Next, we turn to estimating the number of vertices which do not satisfy properties (3) and (4) in the definition of $k$-bad sites.  We will only give a bound on the number of sites which do not satsify property (3), as the argument will be similar for property (4). Note that on the event $\mathcal{B}_1$, all edges between a site $w$ with $\ell(w)\leq k/16$ and some $v$ with $\ell(v)\geq k/3$ are present independently with probability $\frac{1}{2}$. Let $\mathcal{B}_{3}^{v}$ be the event that 
\begin{equation*}
\mathcal{B}_{3}^{v} := \left\{ \deg\left(v, \left[0,\dots,\frac{k}{20}\right]\right)  \notin \left[\frac{k}{80}, \dots, \frac{3k}{80} \right]  \right\} \, .
\end{equation*}
Observe that the events $(\mathcal{B}_{3}^{v})$ are independent, have the same probability, and thus
\begin{equation*}
\mathcal{B}_3 := \left\{ \left| v\in V \colon \ell(v)\in [k/3,k^{11/10}] \text{ and } \mathcal{B}_{3}^{v} \text{ holds }  \right| \leq k^{1/5} \right\}
\end{equation*} satisfies
\begin{equation*}
\P_n(\mathcal{B}_3 \, | \, \mathcal{B}_1) \geq 1- \binom{k^2}{k^{\frac{1}{5}}} \P_{n}\Big( \mathcal{B}_{3}^{v} \text{  for all } v \text{ with } \ell(v)\in \Big[\frac{k}{3},\frac{k}{3}+k^{1/5}\Big] \Big) \geq 1-  \exp\big(-c_3 k^{11/10}\big)
\end{equation*} for some constant $c_3>0$. Taking a union bound over the above events, we conclude.
%
%
%
\end{proof}

\subsection{Non-collider edges between good sites}\label{sec:ConnectionsGoodDeep}

We now give a bound on the set $\mathcal{S}$ of non-collider edges with both endpoints being $k$-good and labels between $k/3$ and $k^{7/6}$. Set
\begin{equation}\label{def:BadSetS}
\mathcal{S} =\mathcal{S}_{k,G}:= \left\{ (v,w) \in C(G) \colon \ell(v),\ell(w) \in \{ k/3,\dots,k^{7/6}\} \text{ and } v,w \text{ are } k\text{-good }\right\} 
\end{equation} for all $k>0$. The non-collider edges between $k$-good and $k$-bad sites will be handled separately in the upcoming section.

\begin{lemma}\label{lem:GoodToGood}
There exist constants $c,N>0$ such that for all $n \geq N$, and $k\in [n^{19/20}]$,
\begin{equation}\label{eq:ManyUpperEdges}
\P_{n}\left( |\mathcal{S}_{k,G}| \geq k^{4/5} \right) \leq \exp(-c k^{11/10}) \, .
\end{equation}
\end{lemma}
\begin{proof} In order to control the total number of non-collider edges with both endpoints of labels between $k/3$ and $k^{7/6}$, fix some set $S$ with
\begin{equation}
S = \left\{ \{a,b\} \colon a,b \in V(G) \text{ and } \ell(a),\ell(b) \in \{k/3,\dots,k^{7/6} \} \right\}
\end{equation} and $|S|=k^{4/5}$. In order to show \eqref{eq:ManyUpperEdges}, it suffices prove for $n$ and $k$ large that
\begin{equation}\label{eq:FixedSetS}
\P_{n}\big( \{ v,w\} \in \mathcal{C}(E) \colon \ell(v)=i \text{ and } \ell(w)=j \text{ for all } (i,j)\in S  \big) \leq \exp(-c k^{11/10})
\end{equation} for any such fixed subset $S$,  taking a union bound over the at most $\binom{k^{7/6}}{k^{4/5}}$ choices for such sets $S$. In the following, for each $k$-good site $u$, we denote by $\mathcal{U}_1(u)$ the set of sites according to property (3), which have an edge to $u$, and by $\mathcal{U}^{\textup{c}}_1(u)$ the set of sites $v$ with $\ell(v)\in [k/20]$ such that $v \notin \mathcal{U}_1(u)$, i.e.\ there is no edge between $v$ and $u$. 
Depending on the structure of the set $S$, we distinguish three cases.
First, assume there exists some site $s$ and a set of sites $W(s)$ with $|W(s)| \geq k^{2/3}$ such that $\{s,u\} \in S$ and $\ell(u) \geq \ell(s)$ for all $u\in W(s)$. Since $s$ is $k$-good, we see that  all edge of the form $(v,u)$ for some $v\in \mathcal{U}^{\textup{c}}_1(s)$ and $u\in W(s)$ must be absent in order to have $S$ as a set of non-collider edges. 
Since $|\mathcal{U}^{\textup{c}}_1(u)|\geq k/80$ and $|W(s)| \geq k^{2/3}$, this allows us to conclude \eqref{eq:FixedSetS} for this choice of $S$. 
Next, suppose that there exists some site $s$ and a set of sites $W^{\prime}(s)$ with $|W^{\prime}(s)| \geq k^{2/3}$ such that $\{u,s\} \in S$ and $\ell(u) \leq \ell(s)$ for all $u\in W^{\prime}(s)$. In order to have $S$ as a set of non-collider edges, all edges $(u,w)$ for sites $w\in W^{\prime}(s)$ and $u\in \mathcal{U}_1(s)$ must be present, allowing us to conclude 
\eqref{eq:FixedSetS} for this choice of $S$. 
Finally, if none of the above two cases holds, there exists a subset $W \subseteq S$ of disjoint edges with $|W| \geq k^{1/10}$. For each of these edges $(v,w)$, we must have that all edges of the form $(v,u)$ for some $u\in \mathcal{U}_1(w)$ are present. Since $|\mathcal{U}_1(w)|\geq k/80$, and there are at at least $k^{1/10}$ choices for $v$, this allows us to conclude 
\eqref{eq:FixedSetS} for this choice of $S$, and thus finishes the proof.
\end{proof}

\subsection{Non-collider edges connecting to good sites}

It remains to control the number of sites $v$ which connect to some $k$-good site by a non-collider edge.

\begin{lemma}\label{lem:ConnectionGoodBad}
There exists some universal constant $c>0$ such that for all $k\in \N$, with probability at least $1-\exp(-c k^{\frac{11}{10}})$, there are at most $k^{1/10}$ sites $v$ which are connected to some $k$-good site by a non-collider edge.
\end{lemma}
\begin{proof}
Observe that the property of being $k$-good does not depend on the edges having their lower endpoint in the bottom $k/6$ sites, and their upper endpoint between $k/6$ and $k/3$. In the following, we partition the sites $v$ with $\ell(v) \leq k^{7/6}$ as
\begin{align*}
\mathcal{S}_1 &:= \left\{ v\in V \colon \ell(v) \in [k/6] \right\} \\
\mathcal{S}_2 &:= \left\{ v\in V \colon \ell(v) \in \{k/6+1,\dots,k/3\} \right\} \\
\mathcal{S}_3 &:= \left\{ v\in V \colon \ell(v) \in \{k/3+1,\dots,k^{7/6}\}  \right\} .
\end{align*}
For each $k$-good site $w$, recall that $\mathcal{U}_1(w)$ denotes the set of sites according to property (3), which have an edge to $w$, and that $\mathcal{U}^{\textup{c}}_1(w)$ is the set of sites $v$ with $\ell(v)\in [k/20]$ such that $v \notin \mathcal{U}_1(w)$. Similarly, we let $\mathcal{U}_2(w)$ be the set of sites according to property (4), which have an edge to $w$, and $\mathcal{U}^{\textup{c}}_2(w)$ the set of sites $v$ with $\ell(v)\in [k/5,k/4]$ such that $v \notin \mathcal{U}_2(w)$. We argue in the following for the sites in $\mathcal{S}_i$  separately for each $i\in [3]$. \\

For each site $v \in \mathcal{S}_1$, we define the event
\begin{equation*}
\mathcal{A}^1_v := \left\{ \exists \, k\text{-good } w \colon (v,w) \in E \text{ and } (u,v) \in E \text{ for all } u \in \mathcal{U}_1(w) \, \right\}  .
\end{equation*} Note that $\mathcal{A}^1_v$ is a necessary condition to have a non-collider edge $(v,w)$ for some $k$-good site $w$. Recall that $|\mathcal{U}_1(w)|\geq k/80$ when $w$ is $k$-good. Conditioning on the event $\mathcal{B}_1$, using the independence of the presence of these edges and a union bound over the sites $w$, we get
\begin{equation*}
 \P_{n}(\mathcal{A}^1_v \, | \, \mathcal{B}_1 ) \leq k^2 2^{-k/80} 
\end{equation*} for all $n\in \N$ sufficiently large. Since for a fixed set of $k$-good sites, the events $\mathcal{A}^1_v$ are independent, we see that for all $n\in \N$ sufficiently large, and all $k \in [n^{19/20}]$
\begin{equation*}
\P_{n}\left( | v\in \mathcal{S}_1 \colon (v,w) \text{ non-collider for some } k\text{-good } w | \geq \frac{1}{4}k^{\frac{1}{10}} \right) \leq \exp(-c_1 k^{11/10})
\end{equation*} for some suitable constant $c_1>0$. Next, we define for each site $v \in \mathcal{S}_2$ the event 
\begin{equation*}
\mathcal{A}^2_v := \left\{ \exists \, k\text{-good } w \colon (v,w) \in E \text{ and } (v,u) \in E \text{ for all } u \in \mathcal{U}_2(w) \, \right\}  .
\end{equation*} Similarly, we get that
\begin{equation*}
 \P_{n}(\mathcal{A}^2_v \, | \, \mathcal{B}_1 ) \leq k^2 2^{-k/80} \, .
\end{equation*} for all $n\in \N$ sufficiently large, and all $k \in [n^{19/20}]$. Since the events $\mathcal{A}^2_v$ are independent, and a necessary condition to have a non-collider edge $(v,w)$ for some $k$-good $w$, we obtain
\begin{equation*}
\P_{n}\left( | v\in \mathcal{S}_2 \colon (v,w) \text{ non-collider for some } k\text{-good } w | \geq \frac{1}{4}k^{\frac{1}{10}} \right) \leq \exp(-c_2 k^{11/10})
\end{equation*} for some suitable constant $c_2>0$. It remains to consider the case where $v\in \mathcal{S}_3$. We define 
\begin{equation*}
\mathcal{A}^3_v := \left\{ \exists \, k\text{-good } w \colon \ell(w) > \ell(v) \text{ and } (v,w) \in E \text{ and } (u,v) \in E \text{ for all } u \in \mathcal{U}_1(w) \, \right\} \, ,
\end{equation*} and obtain that for some constant $c_3>0$, all $n\in \N$ sufficiently large, and all $k \in [n^{19/20}]$
\begin{align*}
 \P_{n}&\left( | v\in \mathcal{S}_3 \colon (v,w) \text{ non-collider for some } k\text{-good } w \text{ with } \ell(w)\geq \ell(v)| \geq \frac{1}{4}k^{\frac{1}{10}} \right) \\ &\leq \exp(-c_3 k^{11/10})  ,
\end{align*} using a similar argument as in the previous two cases. For the connection between $k$-bad sites in $\mathcal{S}_3$ with respect to the remaining $k$-good sites below them, we define
\begin{equation*}
\mathcal{A}^4_v := \left\{ \exists \, k\text{-good } w \colon \ell(w) < \ell(v) \text{ and }(v,w) \in E \text{ and } (u,v) \notin E \text{ for all } u \in \mathcal{U}^{\textup{c}}_1(w) \, \right\} \, .
\end{equation*} Note that the events $\mathcal{A}^4_v$ are independent for all $v\in \mathcal{S}_3$, and a necessary condition for having a non-collider edge,  using property  (3) of the $k$-good vertices. Since $|\mathcal{U}^{\textup{c}}_1| \geq k/80$ for all $k$-good sites, we get that for all $n\in \N$ sufficiently large
\begin{equation*}
 \P_{n}(\mathcal{A}^4_v \, | \, \mathcal{B}_1 ) \leq k^2 2^{-k/80} \, .
\end{equation*} Hence, we obtain that for some suitable constant $c_4>0$ and for all $n\in \N$ large enough
\begin{align*}
 \P_{n}&\left( | v\in \mathcal{S}_3 \colon (v,w) \text{ non-collider for some } k\text{-good } w \text{ with } \ell(w)< \ell(v)| \geq \frac{1}{4}k^{\frac{1}{10}} \right) \\ &\leq \exp(-c_4 k^{11/10})  .
\end{align*}
Combining the above bound on the sites in $\mathcal{S}_1$, $\mathcal{S}_2$, and $\mathcal{S}_3$ finishes the proof.
\end{proof}

\subsection{Combining the above estimates}\label{sec:ProofProposition}

We have all tools to show Proposition \ref{pro:NonCollider}, and thus to conclude our main result Theorem \ref{thm:Moments}.

\begin{proof}[Proof of Proposition \ref{pro:NonCollider}] We apply Proposition \ref{pro:RefinedMECcounting} in order to estimate the size of the underlying Markov equivalence class. Fix some $k=k(n)\in [n^{19/20}]$. Note that by Lemma~\ref{lem:ColliderHighInTower}, it suffices to consider the case where the vertex set $V(\mathcal{C})$ of the chain graph consists only of sites $w$ with $\ell(w)\leq k^{7/6}$. We declare $V_1$ to be the set of sites $v$ which are $k$-bad sites, have a label $\ell(v)\leq k/3$, or are part of the set $\mathcal{S}_{k,G}$ defined in \eqref{def:BadSetS}. All other sites in $V(\mathcal{C})$ are assigned to $V_2$. Recall the constants $a,b>0$ in Proposition~\ref{pro:RefinedMECcounting}, and set
\begin{align}\label{eq:ChoiceABC}
a := k^{\frac{4}{5}} \quad \text{ and } \quad b := k^{\frac{1}{10}} \, .
\end{align} Note that on the event $\mathcal{B}_1$, which holds with probability at least $1-\exp(-c_1k^{\frac{11}{10}})$ by Lemma~\ref{lem:NumberOfSitesPerGeneration} for some  $c_1>0$, we get the first property in Proposition~\ref{pro:RefinedMECcounting} with respect to $a$ in \eqref{eq:ChoiceABC}. The second property in Proposition~\ref{pro:NonCollider} holds with probability at least $1-\exp(-c_2k^{\frac{11}{10}})$ for some $c_2>0$ using Lemma~\ref{lem:ConnectionGoodBad} for $b$ from \eqref{eq:ChoiceABC}. The third and forth property hold with probability at least $1-\exp(-c_3k^{\frac{11}{10}})$ for some $c_3>0$, using the construction of $k$-bad sites and Lemma~\ref{lem:GoodToGood}. Combining  Lemma~\ref{lem:BadSites} and Lemma~\ref{lem:GoodToGood}. we have with probability at least $1-\exp(-c_4k^{\frac{11}{10}})$ for some $c_4>0$ that $|V_1|\leq k/2$. In total, we get 
\begin{equation}
|M(G)| \leq (k/2)! \Big(k^{7/6}\Big(10k^{4/5}\Big)!\Big)^{k^{1/10}} \leq k !
\end{equation} with probability at least $1-\exp(-ck^{\frac{11}{10}})$ for all $k>0$ sufficiently large, and some constant $c>0$. This finishes the proof of Proposition \ref{pro:NonCollider}, and thus of Theorem \ref{thm:Moments}.
\end{proof} 

\subsection*{Acknowledgment} We thank Ralph Adolphs and Frederick Eberhardt for fruitful discussions at a summer school on \textit{Causality in Neuroscience} organized by the \textit{Studienstiftung des deutschen Volkes}, and Dor Elboim for helpful comments. DS acknowledges the DAAD PRIME program for financial support.  AS was partially supported by NSF grant DMS-1855527, a Simons Investigator grant and a MacArthur Fellowship.

\nocite{*}
\bibliographystyle{plain}
\bibliography{DAGMEC.bib}

\end{document}